\documentclass{amsart}

\usepackage{indentfirst}
\usepackage{amssymb}
\usepackage{amsmath}
\usepackage{graphics}
\usepackage{epsfig}
\usepackage{xcolor}
\usepackage{MnSymbol,wasysym}

\usepackage[utf8]{inputenc}
\usepackage{latexsym}
\usepackage{graphicx}

\usepackage[normalem]{ulem}
\usepackage{soul}

\catcode`@=11 \@addtoreset{equation}{section}
\renewcommand\theequation{\thesection.\@arabic\c@equation}
\catcode`@=12

\newcommand{\RR}{\mathbb{R}}

\expandafter\chardef\csname pre amssym.def
at\endcsname=\the\catcode`\@ \catcode`\@=11
\def\undefine#1{\let#1\undefined}
\def\newsymbol#1#2#3#4#5{\let\next@\relax
 \ifnum#2=\@ne\let\next@\msafam@\else
 \ifnum#2=\tw@\let\next@\msbfam@\fi\fi
 \mathchardef#1="#3\next@#4#5}
\def\mathhexbox@#1#2#3{\relax
 \ifmmode\mathpalette{}{\m@th\mathchar"#1#2#3}%
 \else\leavevmode\hbox{$\m@th\mathchar"#1#2#3$}\fi}
\def\hexnumber@#1{\ifcase#1 0\or 1\or 2\or 3\or 4\or 5\or 6\or 7\or 8\or
 9\or A\or B\or C\or D\or E\or F\fi}

\font\teneufm=eufm10 \font\seveneufm=eufm7 \font\fiveeufm=eufm5
\newfam\eufmfam
\textfont\eufmfam=\teneufm \scriptfont\eufmfam=\seveneufm
\scriptscriptfont\eufmfam=\fiveeufm

\catcode`\@=\csname pre amssym.def at\endcsname

\newcommand{\eqn}{\begin{eqnarray}}
\newcommand{\een}{\end{eqnarray}}

\newtheorem {Theorem}  {Theorem}

\numberwithin{Theorem}{section}
\newtheorem{Lemma}[Theorem]{Lemma}

\newtheorem{Definition}[Theorem]{Definition}
\theoremstyle{Remark}
\newtheorem{Remark}[Theorem]{Remark}

\newcommand{\CC}{{\mathbb C}}

\renewcommand{\a}{\alpha}

\renewcommand{\div}{\mbox{div}}

\newcommand{\Mm}{{\mathcal M}}

\newcommand{\bb}{\mbox{\boldmath $b$}}
 \newcommand{\ww}{\mbox{\boldmath $w$}}
\newcommand{\uu}{\mbox{\boldmath $u$}}
\newcommand{\zz}{\mbox{\boldmath $z$}}

\begin{document}

\title[Decay for 3D magneto-micropolar fluids]{Sharp decay estimates and asymptotic behaviour for 3D magneto-micropolar fluids}

\author[C. J. Niche]{C\'esar J. Niche}
\address[C.J. Niche]{Departamento de Matem\'atica Aplicada, Instituto de Matem\'atica. Universidade Federal do Rio de Janeiro, CEP 21941-909, Rio de Janeiro - RJ, Brazil}
\email{cniche@im.ufrj.br}

\author[C. F. Perusato]{Cilon F. Perusato}
\address[C. F. Perusato]{Departamento de Matem\'atica. Universidade Federal de Pernambuco, 
 CEP 50740-560, Recife - PE. Brazil}
\email{cilon@dmat.ufpe.br}

\thanks{C.J. Niche acknowledges support from Bolsa PQ CNPq - 308279/2018-2  and PROEX - CAPES. C.J. Niche and C.F. Perusato acknowledge support from PRONEX-FAPERJ ``Equa\c{c}\~oes Diferenciais Parciais N\~ao Lineares e Aplica\c{c}\~oes''. C. F. Perusato was partially supported by CAPES--PRINT - 88881.311964/2018--01 and Propesq-UFPE - 08-2019 (Qualis A). He is also grateful for the warm hospitality during his visit at the Universidade Federal do Rio de Janeiro, where this work was started.} 



\date{\today}

\begin{abstract}
We characterize the $L^2$ decay rate of solutions to the 3D magneto-micropolar system in terms of the decay character of the initial datum. Due to a linear damping term, the micro-rotational field has a faster decay rate. We also address the asymptotic behaviour of solutions by comparing them to solutions to the linear part. As a result of the linear damping, the difference between the micro-rotational field and its linear part also decays faster. As part of the proofs of these results, we prove estimates for the derivatives of solutions which might be of independent interest. 
\end{abstract}

\maketitle

\section{Introduction}

The Navier-Stokes equations are one of the main tools for the mathematical study of the evolution of incompressible, homogeneous fluids. When the fluid has more properties or structure arising from the physical model studied, it is necessary to couple these equations to others describing the new features. Recently, there has been a surge of activity on the study of the magneto-micropolar system

\begin{equation}
\label{eqn:MHDmicropolar}
\left\{
\begin{aligned}    
\partial_t \uu +( \uu \cdot \nabla) \uu + \nabla p& =  (\mu + \chi) \Delta \uu + \chi \nabla \times \ww + (\bb \cdot \nabla)  \bb,  \\
\partial_t \ww + (\uu \cdot \nabla) \ww & =  \gamma \Delta \ww +  \nabla (\nabla \cdot \ww) + \chi \nabla \times \uu -2 \chi \ww,  \\
\partial_t \bb + (\uu \cdot \nabla) \bb & =  \nu \Delta \bb + (\bb \cdot \nabla) \uu, \\
\nabla \cdot \uu (\cdot,t) & =  \nabla \cdot \bb(\cdot,t)\, = 0,
\end{aligned}
\right.
\end{equation}
with initial data 
$ \zz_0 =(\:\!\uu_0, {\bf w}_0, \bb_0) \in L^2 _{\sigma}(\mathbb{R}^{3}) \!\times\! L^2 (\mathbb{R}^{3}) \times L^2 _{\sigma}(\mathbb{R}^{3}) $. From now on, we denote $\zz = (\uu,\ww,\bb) $. This system, introduced by Ahmadi and Shahinpoor \cite{MR0443550} to study stability of solutions to (\ref{eqn:MHDmicropolar})  in bounded domains (see also Galdi and Rionero \cite{MR0467030}), models the evolution in time of a $3D$ homogeneous, conducting, incompressible fluid with velocity $\uu$, pressure $p$ and magnetic field $\bb$, which possesses some ``microstructure'' described by a micro-rotational velocity $\ww$. This microstructure may correspond to rigid microparticles suspended or diluted in the fluid, as may be the case for liquid crystals or polimer solutions. The positive constants $\mu, \gamma$ in (\ref{eqn:MHDmicropolar}) correspond to the kinematic and angular viscosity respectively,  $\nu$ is the inverse of the magnetic Reynolds number and $\chi$ is the micro-rotational viscosity. Note that (\ref{eqn:MHDmicropolar}) reduces to the Navier-Stokes equations, when $\bb \equiv \ww \equiv 0$; to the MHD system, when $\ww \equiv 0$; and to the micropolar system, when $\bb \equiv 0$.

Equations (\ref{eqn:MHDmicropolar}) were introduced by Ahmadi and Shahinpoor \cite{MR0443550}, who based their model on the theory of micropolar fluids developed by Eringen \cite{MR0204005} and  studied stability of solutions in bounded domains (see also Galdi and Rionero \cite{MR0467030}). In this context of  bounded domains many results have been obtained concerning different aspects of the study of (\ref{eqn:MHDmicropolar}), as existence of weak and strong solutions (Boldrini, Dur\'an and Rojas-Medar \cite{MR2646523}, Boldrini and Rojas-Medar \cite{MR1666509}, Ortega-Torres and Rojas-Medar \cite{MR1810322}, Rojas-Medar \cite{MR1484679}), stability or blowup of solutions (Braz e Silva, Friz e Rojas-Medar \cite{MR3516831}, Mallea-Zepeda and Ortega-Torres \cite{MR3995955}, Melo \cite{MR3429636}), asymptotic behaviour (Lukasiewicz and Sadowski \cite{MR2047286}, Orlinski \cite{MR3125139}, Sadowski \cite{MR1960745}, Yamaguchi \cite{MR2158216}), numerical methods for (\ref{eqn:MHDmicropolar}) (Ortega-Torres, Rojas-Medar and Cabrales \cite{MR2879801}, Rojas-Medar \cite{MR1479160}) and properties of stochastic versions of (\ref{eqn:MHDmicropolar}) (Yamazaki \cite{MR3622437}, \cite{MR3810101}, \cite{MR3903776}, \cite{MR4034672}). 

Plenty of results have been obtained for (\ref{eqn:MHDmicropolar}) in $\RR ^3$, as the problems studied and the techniques used to solve them are inspired on those for the Navier-Stokes equations. Amongst the many articles on the magneto-micropolar system recently published, we should mention those on existence of weak and strong solutions in different function spaces (Ma \cite{MR3758697}, Yuan \cite{MR2419091}, Wang and Wang \cite{MR3543126}), Beale-Kato-Majda criteria and blowup results (Braz e Silva, Melo and Zingano \cite{MR3592792}, Gala, Sawano and Tanaka \cite{MR2945855}, Wang \cite{MR3369594}, Wang, Hu and Wang \cite{MR2775733}, Wang, Li and Wang \cite{MR2834313}, Zhang and Zhao \cite{MR3001674}), regularity criteria (Gala \cite{MR2639150}, Guo, Zhang and Wang \cite{MR2927090}, Wang \cite{MR3041770}, Wang and Gu \cite{MR3990123}, Yuan \cite{MR2778615}, Yuan and Li \cite{MR4008703}, Xiang and Yang \cite{MR3017165}, Zhang, Yao and Wang \cite{MR2781751}), and  properties of stochastic versions of (\ref{eqn:MHDmicropolar}) (Yamazaki \cite{MR3310628}).

In this article, we are mainly concerned with the $L^2$ norm decay and asymptotic behaviour of solutions to  (\ref{eqn:MHDmicropolar}). Guterres, Nunes and Perusato \cite{MR3853142} proved that the norm of Leray solutions tends to zero, i.e. for $\zz_0 \in L^2$

\begin{equation}
\label{eqn:decay-to-zero}
\lim _{t \to \infty} \Vert \zz(t) \Vert _{L^2} = 0.
\end{equation}
Moreover, when $\chi > 0$, they obtained a sharper result for the micro-rotational field $\ww$, namely

\begin{equation}
\label{eqn:decay-w-to-zero}
\lim _{t \to \infty} t ^{\frac{1}{2}}\Vert \ww(t) \Vert _{L^2} = 0.
\end{equation}
For initial data $\zz _0 \in \left( L^1 (\RR^3) \cap L^2 (\RR^3) \right) ^3$, Li and Shang \cite{MR3825173} used the classical Fourier Splitting method to prove that the decay has algebraic rate, i.e.

\begin{equation}
\label{eqn:li-shang}
\Vert \zz (t) \Vert _{L^2} ^2 \leq C (1 + t) ^{- \frac{3}{2}}.
\end{equation}
With the same initial data, Cruz and Novais \cite{CruzNovais} proved that the decay rate for the micro-rotational field can be improved to

\begin{equation}
\label{eqn:cruz-novais}
\Vert  \ww(t) \Vert _{L^2} ^2 \leq C (1 + t) ^{- \frac{5}{2}}.
\end{equation}
Using a method based on estimates for decay of equations on Sobolev spaces with negative indices, Tan, Wu and Zhou \cite{MR3912713} proved that for initial data $\zz_0$ which is small in $H^N (\RR^3)$, for $N \in \mathbb{Z}$, $N \geq 3$ and which also belongs to either $\dot{H}^{-s} (\RR^3)$ or $\dot{B} ^{-s} _{2, \infty} (\RR^3)$, for $0 \leq s < \frac{3}{2}$, then

\begin{equation}
\label{eqn:besov}
\Vert \zz (t) \Vert _{L^2} ^2 \leq C (1 + t) ^{- s}.
\end{equation}
As a Corollary of this result, they proved that if $\zz_0 \in H^N (\RR^3) \cap L^p (\RR^3)$, with $1 \leq p \leq 2$, $N \geq 2$ and small $H^N$ norm, then

\begin{equation}
\label{eqn:corollary-tan-wu-zhou}
\Vert \zz (t) \Vert _{L^2} ^2 \leq C (1 + t) ^{- \frac{3}{2} \left(\frac{2}{p} - 1 \right)}.
\end{equation}
Note that when $p = 1$ this recovers the result in (\ref{eqn:li-shang}), albeit the conditions imposed on the initial data to obtain (\ref{eqn:corollary-tan-wu-zhou}) are stronger.

\begin{Remark} For results concerning decay of other norms, or of norms of derivatives of solutions, see Guterres, Nunes and Perusato \cite{MR3853142}, Perusato, Melo, Guterres and Nunes \cite{doi:10.1080/00036811.2019.1578347} and Tan, Wu and Zhou \cite{MR3912713}.
\end{Remark}

Our main goal in this article is to improve estimates (\ref{eqn:decay-to-zero}) - (\ref{eqn:corollary-tan-wu-zhou}) by either proving sharper results or by disposing of unnecessary hypotheses, using an unified approach. The main tools we use  are the Fourier Splitting Method and the Decay Character of initial data. The Fourier Splitting Method was devoloped by M.E. Schonbek \cite{MR571048}, \cite{MR775190}, \cite{MR837929} to prove that the $L^2$ norm of solutions to viscous conservation laws and to Navier-Stokes equations decay with algebraic rate when initial data has the form of that that leads to bounds as in (\ref{eqn:li-shang}) and (\ref{eqn:corollary-tan-wu-zhou}). The Decay Character was introduced by Bjorland and M.E. Schonbek \cite{MR2493562} and refined by Niche and M.E. Schonbek \cite{MR3355116} and Brandolese  \cite{MR3493117} and associates to initial data $\zz_0$ in $L^2$ a number $ r^{\ast} = r^{\ast} (\zz_0)$ which characterizes the decay of solutions to a large family of linear systems which such initial data. This, in turn, allows to prove decay estimates for nonlinear equations. For details concerning the decay character and decay of linear systems, see Section \ref{section2}.

Our first result concerning decay of (\ref{eqn:MHDmicropolar}) is the following.

\begin{Theorem} \label{decay-z} Let $\zz$ be a weak solution to (\ref{eqn:MHDmicropolar}), with $32 \, \chi (\mu + \chi + \gamma) > 1, \nu > 0$. Let $r^{\ast} (\zz_0) = r^{\ast}$ be the decay character of $\zz _0$, with $- \frac{3}{2} < r^{\ast} < \infty$. Then, for all $t > 0$

\begin{displaymath}
\Vert \zz (t) \Vert _{L^2} ^2 \leq C (1 + t) ^{- \min \{ \frac{3}{2} + r^{\ast}, \frac{5}{2}\}}.
\end{displaymath}
For $- \frac{3}{2} < r^{\ast} \leq 1$, we have that
\begin{displaymath}
\Vert \zz (t) \Vert _{L^2} ^2 \geq C (1 + t) ^{- \left( \frac{3}{2} + r^{\ast} \right)}. 
\end{displaymath}

\end{Theorem}

As computed in Example 2.6 in Ferreira, Niche and Planas \cite{MR3565380}, for $\zz_0 \in L^p (\RR^3) \cap L^2 (\RR ^3)$ for $1 \leq p < 2$, we have that $r^{\ast} (\zz_0) = - 3 \left( 1 - \frac{1}{p} \right)$. Then, through Theorem \ref{decay-z} we recover the estimates (\ref{eqn:li-shang}) and (\ref{eqn:corollary-tan-wu-zhou}). Note, however, that Tan, Wu and Zhou  \cite{MR3912713} need $\zz_0$ to be small in some Sobolev space for (\ref{eqn:corollary-tan-wu-zhou}) to hold, a hypothesis we do not need in our result. Guterres, Nunes and Perusato \cite{MR3853142} proved that the norm of Leray solutions go to zero, see (\ref{eqn:decay-to-zero}). In Theorem \ref{decay-z} we are able to provide a rate for this decay, as long as the initial datum obeys $- \frac{3}{2} < r^{\ast} < \infty$. As a consequence of the results by Brandolese \cite{MR3493117},  our result also extends and improves estimate (\ref{eqn:besov}) by disposing of the small norm in $H^N$ hypothesis and also by extending the range for which it is valid to $0 \leq s \leq \frac{5}{2}$. We discuss this fact in Section \ref{brandolese}.

The equation for $\ww$ in (\ref{eqn:MHDmicropolar}) has a feature that distinguishes it from those for $\uu$ and $\bb$ in that  contains a linear damping term $2 \chi \ww$. Linear equations or systems of that form use to have exponential decay, so we expect this to improve the decay of $\ww$ with respect to that in Theorem \ref{decay-z}.

\begin{Theorem} \label{decay-w} 
Consider the same hypothesis as in Theorem \ref{decay-z}. Then, we have the improved decay estimate
\begin{displaymath}
\Vert \ww (t) \Vert ^2 _{L^2}  \leq C (1 + t) ^{- \min \{ \frac{5}{2} + r^{\ast}, \frac{7}{2}\}},
\end{displaymath}
for all $ t>0 $.
\end{Theorem}

Thus, the result in this Theorem improves and extends the decays in  (\ref{eqn:decay-w-to-zero}) by Guterres, Nunes and Perusato \cite{MR3853142} and in (\ref{eqn:cruz-novais}) by Cruz and Novais \cite{CruzNovais}. Note that for any algebraic decay rate, by Theorem \ref{decay-z} we can always find initial data with appropiate $r^{\ast}$ that leads to a solution $\zz$ with  decay slower than this  given one. However, from Theorem \ref{decay-w} the decay of $\Vert \ww (t) \Vert _{L^2}$ will be at least of order $(1 + t) ^{- \frac{1}{2}}$ for any initial datum $\ww _0$. This is a consequence of the exponential decay of the linear part of the equation for $\ww$ caused by the  linear damping.  

We now address first order asymptotics, by studying the decay of the difference between the full solution $\zz (t)$ and $\bar{\zz} (t)$, the solution of its linear part
 
\begin{equation*}
\left\{
\begin{aligned}    
\bar{\uu}_t & = (\mu + \chi) \Delta \bar{\uu} +  \chi \nabla \times \bar{\ww},  \\
\bar{\ww}_t & = \gamma \Delta \bar{\ww} +  \nabla  (\nabla \cdot \bar{\ww}) +  \chi \nabla \times \bar{\uu} - 2\chi  \bar{\ww},  \\
\bar{\bb}_t & = \nu \Delta \bar{\bb}
\end{aligned}
\right.
\end{equation*} 
with the same initial data.

\begin{Theorem} 
\label{decay-difference} 
Let $\zz$ be a weak solution to (\ref{eqn:MHDmicropolar}) , with $32 \, \chi (\mu + \chi + \gamma) > 1, \nu > 0$. Let $r^{\ast} (\zz_0) = r^{\ast}$ be the decay character of $\zz _0$, with $- \frac{3}{2} < r^{\ast} < \infty$. Then,  
\begin{equation}
\label{error_z}
\Vert \zz (t) - \bar{{\zz}}(t) \Vert _{L^2} ^2 \leq C (1 + t) ^{- \min\{ \frac{7}{2} + 2r^\ast, \frac{5}{2}\} }, \qquad \forall t > 0 
\end{equation} 
and 
\begin{equation}
\label{error_w}
\Vert \ww (t) - \bar{{\ww}}(t) \Vert _{L^2} ^2 \leq C (1 + t) ^{-  \min\{ \frac{9}{2} + 2r^\ast, \frac{7}{2}\}  }, \qquad \forall t > 0.
\end{equation}
\end{Theorem}
As in the previous results, the exponential decay of the linear part of the equation for $\ww$ leads to a faster decay in  the corresponding asymptotic behaviour.

This work is organized as follows. In Section \ref{section2}  we gather all definition and results concerning the decay character and its use for establishing decay for linear systems. More precisely, in Section \ref{decay-character} we define the decay character and state Theorem \ref{characterization-decay-l2}  (from Niche and M.E. Schonbek \cite{MR3355116}) in which sharp upper and lower bounds are proved for ``diagonalizable'' systems. In Section \ref{linear-part-decay-character}, we specifically apply the results from the previous Section to the linear part of (\ref{eqn:MHDmicropolar}). To wit, we first establish a relation between the decay character of $\zz_0$ and those of $\uu_0, \ww_0$ and $\bb_0$. Then, we prove a Lemma that allows us to effectively use Theorem \ref{characterization-decay-l2}. Finally, in Section \ref{brandolese} we carefully discuss the work of Brandolese \cite{MR3493117}, which we use to show that Theorem \ref{decay-z} extends some previously known estimates. In Section \ref{proofs} we prove our results. We point out that some gradient estimates proved in this Section, more specifically Lemmas \ref{gradient-z-decay}, \ref{lemma-estimate-w-and-z} and \ref{lemma-gradient-z-barz}, may be of independent interest.

\section{Decay character and decay of linear part}
\label{section2}

\subsection{Decay character and linear operators} \label{decay-character}

In order to establish sharp decay rates for the linear part in (\ref{eqn:MHDmicropolar}), we recall the idea of decay character, as defined and developed by Bjorland and M.E. Schonbek \cite{MR2493562}, Niche and M.E. Schonbek \cite{MR3355116} and Brandolese \cite{MR3493117}. 

As the long time evolution of the norm of solutions is determined by its low frequencies, it is expected that the small frequencies of the initial datum provide insight into the decay of the $L^2$ or Sobolev norms of  linear systems. Roughly speaking, the decay character compares $ |\widehat{v_0} (\xi)|^2$ to $f(\xi) = |\xi|^{2r}$ near $\xi = 0$.

\begin{Definition} \label{decay-indicator}
Let  $v_0 \in L^2(\RR^n)$. For $r \in \left(- \frac{n}{2}, \infty \right)$, we define the {\em decay indicator}  $P_r (v_0)$ corresponding to $v_0$ as

\begin{displaymath}
P_r(v_0) = \lim _{\rho \to 0} \rho ^{-2r-n} \int _{B(\rho)} \bigl |\widehat{v_0} (\xi) \bigr|^2 \, d \xi,
\end{displaymath}
provided this limit exists. In the expression above,  $B(\rho)$ denotes the ball at the origin with radius $\rho$.

\end{Definition}

\begin{Definition} \label{df-decay-character} The {\em decay character of $ v_0$}, denoted by $r^{\ast} = r^{\ast}( v_0)$ is the unique  $r \in \left( -\frac{n}{2}, \infty \right)$ such that $0 < P_r (v_0) < \infty$, provided that this number exists. If such  $P_r ( v_0)$ does not exist, we set $r^{\ast} = - \frac{n}{2}$, when $P_r (v_0)  = \infty$ for all $r \in \left( - \frac{n}{2}, \infty \right)$  or $r^{\ast} = \infty$, if $P_r (v_0)  = 0$ for all $r \in \left( -\frac{n}{2}, \infty \right)$.
\end{Definition}

The decay character can be explicitly computed in many cases. For example as pointed out in the Introduction, when $v_0 \in L^p (\RR^n) \cap L^2 (\RR ^n)$ for $1 \leq p < 2$, we have that $r^{\ast} (v_0) = - n \left( 1 - \frac{1}{p} \right)$, see Example 2.6 in Ferreira, Niche and Planas \cite{MR3565380}. For more, see Example 2.5 in Niche and M.E. Schonbek \cite{MR3355116}. 

We now use the decay character for establishing upper and lower bounds for decay rates of energy for solutions to a large family of dissipative linear operators. For a Hilbert space $X$ on $\RR^n$, we consider a pseudodifferential operator $\mathcal{L}: X^n \to \left( L^2 (\RR^n) \right) ^n$, with symbol $ \Mm(\xi)$ such that 

\begin{equation}
\label{eqn:symbol}
\Mm(\xi) = P^{-1} (\xi) D(\xi) P(\xi), \qquad \xi-a.e.
\end{equation}
where $P(\xi) \in O(n)$ and $D(\xi) = - c_i |\xi|^{2\a} \delta _{ij}$, for $c_i > c>0$ and $0 < \a \leq 1$.  Taking the Fourier Transform of the linear equation

\begin{equation}
\label{eqn:linear-part}
v_t = \mathcal{L} v,
\end{equation}
multiplying by $\widehat{v}$, integrating in space and then using (\ref{eqn:symbol}) we obtain

\begin{equation}
\label{eqn:key-inequality}
\frac{1}{2} \frac{d}{dt} \Vert \widehat{v}(t) \Vert _{L^2} ^2 \leq  - C  \int _{\RR^n} |\xi|^{2 \a} |\widehat{v}|^2 \, d \xi,
\end{equation}
which is the key inequality for using the Fourier Splitting method. The vectorial fractional Laplacian and the operator 

\begin{equation}
\label{eqn:lame}
\mathcal{L} u = \Delta u + \nabla \, \div \, u
\end{equation}
obey (\ref{eqn:symbol}), so they are amenable to our analysis, see Examples 2.8 and 2.9 in Niche and M.E. Schonbek \cite{MR3355116}.

We now state the Theorem that describes decay in terms of the decay character for linear operators as in (\ref{eqn:symbol}).

\begin{Theorem}{(Theorem 2.10, Niche and M.E. Schonbek \cite{MR3355116})}
\label{characterization-decay-l2}
Let $v_0 \in L^2 (\RR^n)$ have decay character $r^{\ast} (v_0) = r^{\ast}$. Let $v (t)$ be a solution to  (\ref{eqn:linear-part}) with data $v_0$, where the operator $\mathcal{L}$ is such that (\ref{eqn:symbol})  holds. Then:
\begin{enumerate}
\item if $- \frac{n}{2 } < r^{\ast}< \infty$, there exist constants $C_1, C_2> 0$ such that
\begin{displaymath}
C_1 (1 + t)^{- \frac{1}{\a} \left( \frac{n}{2} + r^{\ast} \right)} \leq \Vert v(t) \Vert _{L^2} ^2 \leq C_2 (1 + t)^{- \frac{1}{\a} \left( \frac{n}{2} + r^{\ast} \right)};
\end{displaymath}
\item if $ r^{\ast}= - \frac{n}{2}$, there exists $C = C(\epsilon) > 0$ such that
\begin{displaymath}
\Vert v(t) \Vert _{L^2} ^2 \geq C (1 + t)^{-\epsilon}, \qquad \forall \epsilon > 0,
\end{displaymath}
i.e. the decay of $\Vert v(t) \Vert _{L^2} ^2$ is slower than any uniform  algebraic rate;
\item if $r^{\ast} = \infty$, there exists $C > 0$ such that
\begin{displaymath}
\Vert v(t) \Vert _{L^2} ^2 \leq C (1 + t) ^{- m}, \qquad \forall m > 0,
\end{displaymath}
i.e. the decay of $\Vert v(t) \Vert _{L^2}$ is faster than any algebraic rate.
\end{enumerate} 
\end{Theorem}

\subsection{Decay characterization for the linear part of (\ref{eqn:MHDmicropolar})} \label{linear-part-decay-character} We now study the linear system associated to (\ref{eqn:MHDmicropolar}), namely    

\begin{equation}
\label{eqn:linear}
\left\{
\begin{aligned}    
\partial_t \bar{\uu} & = (\mu + \chi) \Delta \bar{\uu} +  \chi \nabla \times \bar{\ww},  \\
\partial_t \bar{\ww}& = \gamma \Delta \bar{\ww} +  \nabla  (\nabla \cdot \bar{\ww}) +  \chi \nabla \times \bar{\uu} - 2\chi  \bar{\ww},  \\
\partial_t \bar{\bb} & = \nu \Delta \bar{\bb}
\end{aligned}
\right.
\end{equation} 
with initial data $ \bar{\zz}_0 = \zz_0 = (\uu _0, \ww_0, \bb_0) \in L^2_\sigma (\RR^3) \times L^2 (\RR^3) \times L^2 _{\sigma} (\RR^3)$, where we set $ \bar{\zz} =  (\bar{\uu},\bar{\ww},\bar{\bb}) \subset L^2 (\RR ^9)$, for simplicity. 

We first address the relation between $r^{\ast} (\zz_0), r^{\ast} (\uu_0), r^{\ast} ({\bf w}_0)$ and $r^{\ast} (\bb_0)$.

\begin{Lemma} Let $r^{\ast} (\uu_0), r^{\ast} ({\bf w}_0) , r^{\ast} (\bb_0) \in \left( -\frac{3}{2}, \infty \right)$. Then

\begin{displaymath}
r^{\ast} (\zz _0)  = \min \{ r^{\ast} (\uu_0), r^{\ast} ({\bf w}_0) , r^{\ast} (\bb_0) \}.
\end{displaymath}

\end{Lemma}

\begin{proof}  Let $\lambda = \min \{ r^{\ast} (\uu_0), r^{\ast} ({\bf w}_0) , r^{\ast} (\bb_0) \}$. In order to fix ideas, suppose $\lambda = r^{\ast} (\uu _0)$ and $ \lambda <  r^{\ast} ({\bf w}_0) , r^{\ast} (\bb_0)$. Note that 

\begin{displaymath}
P_{\lambda} (\zz _0) = P_{\lambda} (\uu _0) + P_{\lambda} ({\bf w}_0) + P_{\lambda} (\bb _0)
\end{displaymath}
and that $P_{\lambda} (\uu _0) > 0$. Now

\begin{equation*}
\begin{split}
P_{\lambda} ({\bf w}_0) & = \lim _{\rho \to 0} \rho ^{-( 2 \lambda + 3)} \int _{B(\rho)} \bigl |\widehat{{\bf w}_0} (\xi) \bigr|^2 \, d \xi \\ & = \lim _{\rho \to 0} \rho ^{- 2 \left(  \lambda +  \left( r^{\ast} ({\bf w}_0) - \lambda \right) \right)}  \rho ^{2 \left( r^{\ast} ({\bf w}_0) - \lambda \right)} \rho ^{-3} \int _{B(\rho)} \bigl |\widehat{{\bf w}_0} (\xi) \bigr|^2 \, d \xi \\ & = \lim _{\rho \to 0} \rho ^{2  \left( r^{\ast} ({\bf w}_0)  - \lambda \right)} \rho ^{-( 2 r^{\ast} ({\bf w}_0) + 3)} \int _{B(\rho)} \bigl |\widehat{{\bf w}_0} (\xi) \bigr|^2 \, d \xi \\ & = \lim _{\rho \to 0} \rho ^{2 \left( r^{\ast} ({\bf w}_0)  - \lambda \right)} r^{\ast} ({\bf w}_0) = 0
\end{split}
\end{equation*}
because $r^{\ast} ({\bf w}_0)  > \lambda$. The same argument proves that $P_{\lambda} (\bb_0) = 0$, hence $P_{\lambda} (\zz _0) = P_{\lambda} (\uu _0)$, which leads to the result.
\end{proof}

In order to use Theorem  \ref{characterization-decay-l2} we pass to frequency space, where after taking the Fourier transform of (\ref{eqn:linear}) we obtain

\begin{displaymath}
\partial_t \widehat{\bar{\zz}} = M(\xi) \widehat{\bar{\zz}},
\end{displaymath}
where $M = M(\xi)$ is the matrix of symbols given by

\begin{equation}
\label{eqn:matrix-symbols}
M = \left( \begin{array}{ccc} - (\mu + \chi) |\xi|^2 Id_{3 \times 3} & i \chi R_3 (\xi) & 0 _{3 \times 3}  \\  i \chi    R_3 (\xi) & - ( \gamma |\xi|^2 + 2 \chi)  Id_{3 \times 3} -  \xi_i \xi_j  & 0 _{3 \times 3}  \\ 0 _{3 \times 3} & 0 _{3 \times 3} & -  \nu |\xi|^2 Id_{3 \times 3}    \end{array} \right).
\end{equation}
Here $Id_{3 \times 3}$ and $0 _{3 \times 3}$ are the $3 \times 3$ identity and zero matrizes respectively and $ i  R_3 (\xi)$ is the rotation matrix

\begin{displaymath}
 i  R_3 (\xi) =  i  \left( \begin{array}{ccc} 0 & \xi_3 & - \xi_2 \\ - \xi_3 & 0 & \xi_1 \\ \xi_2 & - \xi_1 & 0 \end{array} \right).
\end{displaymath}As $M$ is self-adjoint, it is diagonalizable and $M(\xi) = P^{-1} (\xi) D(\xi) P(\xi)$, where $P \in U(n)$ and $D(\xi)$ is a diagonal matrix.  To use Theorem \ref{characterization-decay-l2} we would need to compute the eigenvalues of $M$, which is a cumbersome task. Instead, we will prove the following Lemma, which provides an estimate for the largest eigenvalue. This immediately leads to (\ref{eqn:key-inequality}) and allows us to use Theorem \ref{characterization-decay-l2}. 

\begin{Lemma} \label{lemma-eigenvalue}  Let $32 \chi (\mu + \chi + \gamma) > 1$. Then, for $M = M( \xi)$ we have that

\begin{displaymath}
\lambda _{max} (M) \leq - C |\xi|^2, \qquad C =C (\mu, \chi, \gamma, \nu) > 0.
\end{displaymath}

\end{Lemma}

\begin{proof} We follow the ideas in Section 3.1 in Ferreira and Villamizar-Roa \cite{MR2324348}. From the Rayleigh-Ritz Theorem we know that for any Hermitian matrix $M \in \mathcal{M} _n (\CC)$, the inequality 

\begin{displaymath}
R_{M} (v) = v^{\ast} M v \leq \lambda_{\max} (M)
\end{displaymath}
holds, for all $\Vert v \Vert_2 = 1$.

A simple computation shows that the matrix $i R_3 (\xi)$ has spectrum $\sigma (i R_3 (\xi)) = \{- |\xi|, 0, |\xi|  \} $, with associated orthonormal eigenvectors $v_1, v_2, v_3$. With these, we construct the orthonormal basis $\mathcal{B} = \left( b_1, \cdots, b_n \right)$ for $\CC ^9$, where

\begin{eqnarray*}
\mathcal{B}   = \left\{  \frac{1}{2} (v_1, v_1, 0), \frac{1}{2} (v_3, - v_3, 0), \frac{1}{2} (v_2, v_2, 0),
\frac{1}{2} (v_2, - v_2, 0),  \right. \\  \left.  \frac{1}{2} (v_3, v_3, 0), \frac{1}{2} (v_1, - v_1, 0), e_7, e_8, e_9 \right\}
\end{eqnarray*}
where $e_7, e_8, e_9$ are the last three vectors in the canonical base in $\CC ^9$.

We now write $M (\xi) = M_1 (\xi) + M_2 (\xi) + M_3 (\xi)$, where

\begin{displaymath}
M_1 (\xi)  = \left( \begin{array}{ccc} - (\mu + \chi) |\xi|^2 Id_{3 \times 3} & 0 _{3 \times 3} & 0 _{3 \times 3}  \\  0 _{3 \times 3} & - ( \gamma |\xi|^2 + 2 \chi)  Id_{3 \times 3}  & 0 _{3 \times 3}  \\ 0 _{3 \times 3} & 0 _{3 \times 3} & -  \nu |\xi|^2 Id_{3 \times 3}    \end{array} \right),
\end{displaymath}

\begin{displaymath}
M_2 (\xi) = \left( \begin{array}{ccc} 0 _{3 \times 3} & 0 _{3 \times 3}  & 0 _{3 \times 3}  \\ 0 _{3 \times 3} & - \xi_i \xi_j  & 0 _{3 \times 3}  \\ 0 _{3 \times 3} & 0 _{3 \times 3} & 0 _{3 \times 3}    \end{array} \right), \quad
M_3 (\xi) = \left( \begin{array}{ccc} 0 _{3 \times 3} & i \chi R_3 (\xi) & 0 _{3 \times 3}  \\  i \chi    R_3 (\xi) & 0 _{3 \times 3}  & 0 _{3 \times 3}  \\ 0 _{3 \times 3} & 0 _{3 \times 3} & 0 _{3 \times 3}    \end{array} \right).
\end{displaymath}

As the eigenvalues of $M_2 $ are $0$ and $- |\xi|$, we have that $v^{\ast} M_2 (\xi) v \leq 0$. Now take $v = \sum _{v = 1} ^9 c_i b_i$. This leads to

\begin{align}
v^{\ast} M_1 v & = - \frac{1}{2} \left( c_1 ^2 + c_2^2  + c_5 ^2 + c_6^2 \right) ((\mu + \chi)|\xi|^2 + \gamma|\xi|^2 + 2 \chi) \notag \\ & - \frac{1}{2} c_3 ^2 (\mu + \chi)|\xi|^2  - \frac{1}{2} c_4^2 \left( \gamma|\xi|^2 + 2 \chi \right)   - \left( c_7 ^2 + c_8 ^2 + c_9 ^2 \right) \nu |\xi|^2 \notag
\end{align}
and to 

\begin{displaymath}
v^{\ast} M_3 v = - \frac{1}{4} \left( c_1 ^2 + c_6 ^2 \right) |\xi| + \frac{1}{4} \left( c_2 ^2 + c_5 ^2 \right) |\xi|.
\end{displaymath}
Then

\begin{align}
v^{\ast} M_1 v & + v^{\ast} M_3 v  = - \left(c_1^2 + c_6^2 \right) \left( \frac{1}{2} (\mu + \chi + \gamma) |\xi|^2 + \frac{1}{4} |\xi| + \chi     \right) \notag \\ & -  \left(c_2^2 + c_5^2 \right) \left( \frac{1}{2} (\mu + \chi + \gamma) |\xi|^2 - \frac{1}{4} |\xi| + \chi     \right) \notag \\ & - \frac{1}{2} c_3 ^2 (\mu + \chi)|\xi|^2  - \frac{1}{2} c_4^2 \left( \gamma|\xi|^2 + 2 \chi \right)   - \left( c_7 ^2 + c_8 ^2 + c_9 ^2 \right) \nu |\xi|^2  \notag \\ & \leq - \Vert v \Vert _2 ^2 \min \{ (\mu + \chi + \gamma) |\xi|^2 - \frac{1}{2} |\xi| + 2 \chi, (\mu + \chi)|\xi|^2, \gamma|\xi|^2 + 2 \chi, 2 \nu |\xi|^2 \notag  \}.
\end{align}
If $32 \chi (\mu + \chi + \gamma) > 1$, then $(\mu + \chi + \gamma) |\xi|^2 - \frac{1}{2} |\xi| + 2 \chi > 0$ for any $|\xi|$ and the result follows.
\end{proof}

The estimate obtained in Lemma \ref{lemma-eigenvalue} leads to (\ref{eqn:key-inequality}). We can now use Theorem  \ref{characterization-decay-l2} to obtain

\begin{Theorem} \label{linear-part-magnetomicropolar}
Let $\bar{\zz}_0 \in L^2 (\RR^9)$ have decay character $r^{\ast} (\bar{\zz} _0) = r^{\ast}$. Then

\begin{enumerate}
\item if $- \frac{n}{2 } < r^{\ast}< \infty$, there exist constants $C_1, C_2> 0$ such that
\begin{displaymath}
C_1 (1 + t)^{- \left( \frac{3}{2} + r^{\ast} \right)} \leq \Vert \bar{\zz}(t) \Vert _{L^2} ^2 \leq C_2 (1 + t)^{-  \left( \frac{3}{2} + r^{\ast} \right)};
\end{displaymath}
\item if $ r^{\ast}= - \frac{3}{2}$, there exists $C = C(\epsilon) > 0$ such that
\begin{displaymath}
\Vert \bar{\zz} (t) \Vert _{L^2} ^2 \geq C (1 + t)^{-\epsilon}, \qquad \forall \epsilon > 0,
\end{displaymath}
i.e. the decay of $\Vert v(t) \Vert _{L^2} ^2$ is slower than any uniform  algebraic rate;
\item if $r^{\ast} = \infty$, there exists $C > 0$ such that
\begin{displaymath}
\Vert \bar{\zz} (t) \Vert _{L^2} ^2 \leq C (1 + t) ^{- m}, \qquad \forall m > 0,
\end{displaymath}
i.e. the decay of $\Vert \bar{\zz} (t) \Vert _{L^2}$ is faster than any algebraic rate.
\end{enumerate} 
\end{Theorem}

As the symbol matrix (\ref{eqn:matrix-symbols}) is diagonalizable, the linear system (\ref{eqn:linear}) decouples. The fact that  the decay character of $\zz _0$  is the mininum of the decay characters of $\uu _0, \ww _0, \bb _0$ , implies that decay of solutions $\bar{\zz} (t)$  to (\ref{eqn:linear}) is the slowest of the decays of  $\bar{\uu} (t), \bar{\ww} (t), \bar{\bb} (t)$.

\subsection{The work of Brandolese \cite{MR3493117} and estimate (\ref{eqn:besov})} \label{brandolese} The decay character of initial data $v_0 \in L^2 (\RR ^n)$ is used to prove sharp upper and lower bounds for decay of ``diagonalizable'' linear systems, see Theorem \ref{characterization-decay-l2} and its application to (\ref{eqn:linear}) in Theorem \ref{linear-part-magnetomicropolar}. In Definitions \ref{decay-indicator} and \ref{df-decay-character}, introduced by Bjorland and M.E. Schonbek \cite{MR2493562} and extended by Niche and M.E. Schonbek \cite{MR3355116},  the existence of a limit and a positive $P_r (u_0)$ are assumed. However this need not be the case for all of $v_0 \in L^2 (\RR ^n)$. Brandolese \cite{MR3493117} constructed initial data in $L^2$, highly oscillating near the origin, for which the limit in Definition \ref{decay-indicator} does not exist for some $r$. As a result of this, the decay character does not exist. Then, Brandolese gave a slightly different definition of decay character, more general than that in  Definitions \ref{decay-indicator} and \ref{df-decay-character}, but which produces the same result when these hold. 

In this same article, Brandolese proved that the decay character $r ^{\ast}$ (in his more general version) exists for $v_0 \in L^2 (\RR ^n)$ if and only if $v_0$ belongs to a certain subset $\dot{\mathcal{A}} ^{- \left(\frac{n}{2} + r^{\ast} \right)} _{2, \infty} \subset \dot{B} ^{- \left(\frac{n}{2} + r^{\ast} \right)} _{2, \infty}$. Moreover, for diagonalizable linear operators $\mathcal{L}$ as in (\ref{eqn:symbol}), solutions to the linear system (\ref{eqn:linear-part}) with initial data $v_0$ obey

\begin{displaymath}
C_1 (1 + t)^{- \frac{1}{\a} \left( \frac{n}{2} + r^{\ast} \right)} \leq \Vert v(t) \Vert _{L^2} ^2 \leq C_2 (1 + t)^{- \frac{1}{\a} \left( \frac{n}{2} + r^{\ast} \right)},
\end{displaymath}
if and only if the decay character $r^{\ast} = r^{\ast} (v_0)$ exists. This provides a sharp characterization of algebraic decay rates for such systems and provides a key tool for studying decay for nonlinear systems.

Now, let us recall estimate (\ref{eqn:besov}), proved by  Tan, Wu and Zhou \cite{MR3912713}. By taking $s = - \left( \frac{3}{2} + r^{\ast} \right)$, where $-\frac{3}{2} \leq r^{\ast} < 0$, their estimate reads 

\begin{equation}
\label{eqn:specific}
\Vert \zz (t) \Vert _{L^2} ^2 \leq C (1 + t) ^{- \left( \frac{3}{2} + r^{\ast} \right)},
\end{equation}
for $\zz_0 \in H^N (\RR ^3) \cap \dot{B} ^{- \left( \frac{3}{2} + r^{\ast} \right)} _{2, \infty} (\RR^3)$ with small $H^N$ norm, for some $N \geq 3$. As a result of Brandolese's results discussed above, existence of (our version of) the decay character implies that using Theorem \ref{decay-z} we obtain (\ref{eqn:specific}) without the necessity of assuming $\zz_0 \in H^N (\RR^3)$. Moreover, our result shows that (\ref{eqn:specific}) also holds for $\frac{3}{2} \leq s \leq \frac{5}{2}$.

\section{Proofs} 
\label{proofs}

\subsection{Proof of Theorem \ref{decay-z}}

\begin{proof} As is usual when using the Fourier Splitting  we  prove decay for regular enough solutions to an approximate nonlinear problem obtained through spectral cutoff, as in Li and Shang \cite{MR3825173} for the magnetomicropolar equations or through retarded mollifiers (see Cafarelli, Kohn and Nirenberg \cite{MR673830}), as in the case of the micropolar fluid equations, see Braz e Silva, Cruz, Freitas e Zingano \cite{MR3955606}. The decay for weak solutions is obtained through a standard limiting process, for full details see Braz e Silva, Cruz, Freitas e Zingano \cite{MR3955606} for the micropolar fluid equations and pages 267--269 in Lemari\'e-Rieusset \cite{MR1938147} and the Appendix in Wiegner \cite{MR881519} for the Navier-Stokes equations case.

We now proceed formally. As we have seen before

\begin{displaymath}
\partial_t \Vert \zz (t) \Vert ^2 _{L^2} \leq - C \Vert \nabla \zz (t) \Vert ^2 _{L^2}.
\end{displaymath}
Let  $B(t) = \{\xi \in \RR^3: |\xi| \leq g(t) \}$, for a nonincreasing, continuous $g$ with $g(0) = 1$. Then 

\begin{align}
- C \Vert \nabla \zz (t) \Vert ^2 _{L^2} & =  - C \int _{B(t)} |\xi|^2 |\widehat{\zz} (\xi, t)|^2 \, d \xi - C \int _{B(t) ^{c}} |\xi|^2 |\widehat{\zz} (\xi, t)|^2 \, d \xi \notag \\ & \leq  - C \int _{B(t) ^{c}} |\xi|^2 |\widehat{\zz} (\xi, t)|^2 \, d \xi \leq - C g^2 (t)  \int _{B(t) ^{c}}|\widehat{\zz} (\xi, t)|^2 \, d \xi \notag
\end{align}
which leads to

\begin{align}
\label{eqn:main-estimate-fs}
\frac{d}{dt} \left( \exp \left( \int _0 ^t C g^2 (s) \, ds \right)  \Vert \zz (t) \Vert _{L^2} ^2 \right) \leq \;\;\;\;\;\;\;\;\;\;\;\;\;\;\;\; \\ g^2 (t) \left( \exp \left( \int _0 ^t C g^2 (s) \, ds \right) \right) \int _{B(t)} |\widehat{\zz} (\xi, t)| ^2 \, d \xi. \notag
\end{align}
We now need a pointwise estimate for

\begin{displaymath}
\widehat{\zz} (\xi, t) = e^{t M (\xi)} \widehat{\zz _0} (\xi)- \int _0 ^t e^{(t - s) M (\xi)} G(\xi, s) \, ds
\end{displaymath}
where $M$ is as in (\ref{eqn:matrix-symbols}) and
\begin{displaymath}
G(\xi,s) = \mathcal {F} \left( NL (\uu, \ww, \bb) \right) (\xi, s)
\end{displaymath}
for 

\begin{equation}
\label{eqn:nonlinear}
NL (\uu, \ww, \bb) = \left( (\bb \cdot \nabla) \bb - (\uu \cdot \nabla) \uu - \nabla p, - (\uu \cdot \nabla) \ww, - (\uu \cdot \nabla) \bb + (\bb \cdot \nabla) \uu  \right).
\end{equation}
For $F = \uu, \bb$ and $G = \uu, \ww, \bb$, we have that

\begin{displaymath}
\mathcal{F} \left( (F \cdot \nabla) G \right) = \mathcal{F} \left( \nabla (F \otimes G) \right) = i \xi \cdot \left( \widehat{F \otimes G} \right),
\end{displaymath}
so

\begin{equation}
\label{eqn:tensor}
|\mathcal{F} \left( (F \cdot \nabla) G \right) (\xi)| \leq |\xi| \Vert F \Vert _{L^2} \Vert G \Vert _{L^2}.
\end{equation}
By taking divergence in the first equation in (\ref{eqn:MHDmicropolar}) we obtain

\begin{displaymath}
\Delta p = div \, (\bb \cdot \nabla) \bb - div \, (\uu \cdot \nabla) \uu
\end{displaymath}
from which we get

\begin{equation}
\label{eqn:pressure}
- |\xi|^2 \widehat{p} (\xi) = i \sum _{j,k} \xi_j \xi_k \widehat{\bb_j \bb_k} + i \sum _{j,k} \xi_j \xi_k \widehat{\uu_j \uu_k} \leq |\xi| ^2 \left(\Vert \bb(t) \Vert ^2 _{L^2} +  \Vert \uu(t) \Vert ^2 _{L^2} \right).
\end{equation}
Then, from (\ref{eqn:nonlinear}), (\ref{eqn:tensor}) and (\ref{eqn:pressure}) we obtain

\begin{align}
\label{eqn:nonlinear_final}
|G(\xi,t)|  \leq C |\xi| \Vert \zz (t) \Vert ^2 _{L^2}.
\end{align}
Thus,
\begin{eqnarray*}
\left| \int _0 ^t e^{(t -s) M (\xi)} G(\xi,s) \, ds \right| & \leq & C \int _0 ^t e^{- C (t - s) |\xi|^2} \, |\xi| \Vert \zz(s) \Vert _{L^2} ^2 \, ds \nonumber \\ & \leq & C |\xi| \left( \int _0 ^t \Vert \zz(s) \Vert _{L^2} ^2 \, ds \right),
\end{eqnarray*}
where we used the estimate in Lemma \ref{lemma-eigenvalue}.

Suppose now that $\Vert \zz(t) \Vert _{L^2} ^2 \leq C (1 + t) ^{- \alpha}$, for  some $0 \leq \alpha$. As a result of this

\begin{displaymath}
\int _{B(t)} \left( \int _0 ^t e^{(t -s) M (\xi)} G(\xi,s) \, ds \right) ^2 d \xi \leq C |\xi|^5 (1 + t) ^{2 (1 - \alpha)},
\end{displaymath}
which leads, after choosing $g^2 (t) = A (1 +t) ^{-1}$ and for large enough $A > 0$, to

\begin{eqnarray}
\label{eqn:int-ball}
\int _{B(t)} |\widehat{\zz} (\xi, t)|^2 \, d \xi & \leq & C \int _{B(t)} |e^{t M (\xi)} \widehat{\zz _0}|^2 \, d \xi + C \int _{B(t)} \left( \int_0 ^t e^{(t -s) M (\xi)} G(\xi,s) \, ds \right) ^2 \, d \xi \nonumber \\ & \leq & C \Vert  e^{t M (\xi)} \widehat{\zz _0} \Vert _{L^2} ^2 + C g^5 (t) (1 + t) ^{2 (1 - \alpha)} \nonumber \\ & \leq & C (t + 1) ^{- \left( \frac{3}{2} + r^{\ast} \right)} + C (1 + t) ^{- \left( \frac{1}{2} +2 \alpha \right)} \nonumber \\ & \leq & C (t + 1) ^{- \min\{\frac{1}{2}+2 \alpha,\frac{3}{2} + r^{\ast}\}},
\end{eqnarray}
where $r^{\ast} = r^{\ast} (\zz_0)$ and we used Theorem \ref{linear-part-magnetomicropolar} for the decay of the linear part. 
From (\ref{eqn:main-estimate-fs}), (\ref{eqn:int-ball}) and our choice of $g$ we obtain

\begin{equation}
\label{eqn:estimate-cases}
\frac{d}{dt} \left( (t + 1) ^{A}  \Vert \zz (t) \Vert _{L^2} ^2 \right) \leq C (t + 1) ^{A - 1} (t + 1) ^{- \min\{\frac{1}{2}+2\alpha,\frac{3}{2} + r^{\ast}\}}.
\end{equation}
We start with $\alpha=0$, this is the apriori estimate  $\Vert \zz(t) \Vert _{L^2} ^2 \leq C$. In $(\ref{eqn:estimate-cases})$ we consider the two cases  $\frac{3}{2} +r^{\ast} \leq   \frac{1}{2}$ and $  \frac{1}{2}\leq \frac{3}{2} +r^{\ast}$. In the first case, i.e. when $r^{\ast} \leq -1$, we have

\begin{equation}
\label{eqn:decay-as-linear}
\Vert \zz (t) \Vert _{L^2} ^2 \leq C  (t + 1) ^{- \left( \frac{3}{2} + r^{\ast} \right)}.
\end{equation}
In the second case, we obtain 

\begin{displaymath}
\Vert \zz (t) \Vert _{L^2} ^2 \leq C  (t + 1) ^{- \frac{1}{2}},
\end{displaymath}
which is the slower decay. Hence we improved our rate to an exponent at least as large as $\alpha = \frac{1}{2}$. We use this estimate to bootstrap in $(\ref{eqn:estimate-cases})$ and we see we have to separate again the study in two cases, namely $ \frac{3}{2} + r^{\ast} \leq \frac{3}{2}$ and $\frac{3}{2} \leq  \frac{3}{2} + r^{\ast}$. In the first case, which corresponds to  $r^{\ast} \leq 0$, we obtain (\ref{eqn:decay-as-linear}) again. In the second situation, i.e. when $r^{\ast} \geq 0$, we have improved to $\alpha = \frac{3}{2}$. But then

\begin{displaymath}
\int _0 ^t \Vert \zz (s) \Vert _{L^2} ^2 \, ds \leq C,
\end{displaymath}
so 

\begin{displaymath}
\int _{B(t)} \left( \int_0 ^t e^{(t -s) \mathcal{M} (\xi)} G(\xi,s) \, ds \right) ^2 \, d \xi \leq C g^5 (t).
\end{displaymath}
Then (\ref{eqn:int-ball}) becomes

\begin{eqnarray*}
\int _{B(t)} |\widehat{\zz} (\xi, t)|^2 \, d \xi & \leq & C \int _{B(t)} |e^{t M (\xi)} \widehat{\zz _0}|^2 \, d \xi + C \int _{B(t)} \left( \int_0 ^t e^{(t -s) M (\xi)} G(\xi,s) \, ds \right) ^2 \, d \xi \\ & \leq & C \Vert  e^{t M (\xi)} \widehat{\zz _0} \Vert _{L^2} ^2 + C g^5 (t)  \\ & \leq & C (t + 1) ^{- \left( \frac{3}{2} + r^{\ast} \right)} + C (1 + t) ^{- \frac{5}{2}} \leq C (t + 1) ^{- \min \{\frac{3}{2} + r^{\ast}, \frac{5}{2} \} }.
\end{eqnarray*}
Using this in (\ref{eqn:main-estimate-fs}) yields the upper bound for decay.

The reverse triangle inequality leads to

\begin{displaymath}
\Vert \zz (t) \Vert _{L^2} \geq \Vert \bar{\zz} (t) \Vert _{L^2} - \Vert \zz(t) - \bar{\zz} (t) \Vert _{L^2},
\end{displaymath}
where $\bar{\zz} (t)$ is the solution to the linear part of (\ref{eqn:MHDmicropolar}), i.e. system (\ref{eqn:linear}). By Theorem \ref{decay-difference}, we only have upper bounds for the decay of $\Vert \zz(t) - \bar{\zz} (t) \Vert _{L^2}$, so the upper bound we have just proved and Theorem \ref{decay-difference} lead to the lower bound only when the decay of linear part is slower than that of the difference.
\end{proof}

\subsection{Proof of Theorem \ref{decay-w}} Our proof follows that of Theorem 3.2 in Braz, Cruz, Freitas and Zingano \cite{MR3955606}, where an analogous result is proved for the micropolar system for $\zz_0 \in L^1 (\RR^3) \cap L^2 (\RR^3)$.

In the proof of the following Lemmas and Theorems we will need the standard heat kernel estimate in $\RR^3$

\begin{equation}
\label{eqn:heat_estimate1}
\Vert \nabla^{m} e^{t \Delta} f \Vert_{L^q} \leq K \Vert f \Vert_{L^r} \, t^{-\frac{3}{2}\left(\frac{1}{r} - \frac{1}{q} \right) - \frac{m}{2}}, \quad \,\,\forall \; t > 0,   
 \end{equation} 
for $1\leq r \leq q \leq \infty$ (see Kato \cite{Kato1984}). We will also need the following gradient estimate.

 \begin{Lemma} \label{gradient-z-decay} Let $\zz$ be a weak solution to (\ref{eqn:MHDmicropolar}), with $32 \, \chi (\mu + \chi + \gamma) > 1, \nu > 0$. Let $r^{\ast} (\zz_0) = r^{\ast}$ be the decay character of $\zz _0$, with $- \frac{3}{2} < r^{\ast} < \infty$. Then, 

\begin{displaymath}
\Vert \nabla \zz (t) \Vert _{L^2} ^2 \leq C (1 + t) ^{- \min \{ \frac{5}{2} + r^{\ast}, \frac{7}{2}\}}, \qquad \forall t \geq t_0,
\end{displaymath}
{for an appropriate, large enough $t_0 = t_0 (\Vert \zz_0 \Vert _{L^2})$.}
\end{Lemma}

\begin{proof} We follow the ideas in Braz, Cruz, Freitas and Zingano \cite{MR3955606}  and Guterres,  Nunes and Perusato \cite{MR3853142}. Let $ \alpha(r^\ast)=\min\{\frac{3}{2} + r^\ast, \frac{5}{2} \}$. Taking $ \delta>0 $, multiplying \eqref{eqn:MHDmicropolar} by $ (1+t)^{\alpha(r^\ast)+\delta}(\uu,\ww,\bb) $ and  integrating on $ \mathbb{R}^3 \times (t_0,t) $, we obtain

\begin{align}\label{first_order_inequality}
(1+t)^{\alpha(r^\ast) + \delta} \|\zz(t) \|_{L^2}^2 & + 2\min\{\mu,\gamma,\nu\} \intop_{t_0}^t (1+s)^{{\alpha(r^\ast) + \delta}}\|\nabla \zz(s) \|_{L^ 2}^2 ds \notag \\ & \leq 
C \int _{t_0} ^t (1+s)^{\alpha(r^\ast) + \delta - 1} \Vert \zz (s) \Vert _{L^2} ^2 \, ds \notag \\ & \leq C (1 + t) ^{\delta},
\end{align}
where we used Theorem \ref{decay-z} on the right hand side. 

We now use the notation $D_k = \partial _{x_k}, D^2 = \sum _{i,j} D_i D_j$.  Taking $D_k$ in the first three equations in \eqref{eqn:MHDmicropolar}, multiplying by $ (1+t)^{\alpha(r^\ast) + \delta +1} (D_k\uu,D_k \ww,D_k\bb)$ and summing up, after integrating in $\RR^3 \times (t_0, t)$ we obtain
\begin{align}
\label{eqn:first-estimate-long}
\begin{aligned}
(1+t)^{\alpha(r^*) + \delta +1 } \Vert \nabla \zz  \|_{L^2}^2  + 2(\mu + \chi) \int_{t_0}^{t}(1+s)^{\alpha(r^*) + \delta + 1} \|D^2 \uu (s)\|_{L^2}^2 ds \\  + 2\,\gamma\int_{t_0}^{t}(1+s)^{\alpha(r^*) + \delta + 1} \|D^2 \ww (s)\|_{L^2}^2 ds  + 2\,\nu\int_{t_0}^{t}(1+s)^{\alpha(r^*) + \delta + 1} \|D^2 \bb (s)\|_{L^2}^2 ds \notag \\  + 2 \int_{t_0}^{t} (1+s)^{\alpha(r^*) + \delta + 1} \|\nabla  \,(\nabla\cdot \ww) (s)\|_{L^2}^2 ds   + 4 \chi \int_{t_0}^{t} (1+s)^{\alpha(r^*) + \delta + 1} \|\nabla \ww(s) \|_{L^2}^2 ds \notag \\
  = \int _{t_0}^t (1+s)^{\alpha(r^\ast) +\delta } \| \nabla \zz (s) \|_{L^2}^2 ds   + 4 \chi \sum _{k = 1} ^3 \int _{t_0} ^t (1+s)^{\alpha(r^*) + \delta + 1}  \langle \nabla \times D_k \uu  (s), D_k \ww (s) \rangle \, ds \notag \\  + 2 \int _{t_0} ^t (1+s)^{\alpha(r^*) + \delta + 1} \int _{\RR^3} \sum _{i,j,k} D_j D_k \uu _i (x,s) \cdot D_k \left( \uu_j (x,s) \uu_i (x,s) \right) \, dx \, dx \notag \\  - 2 \int _{t_0} ^t (1+s)^{\alpha(r^*) + \delta + 1} \int _{\RR^3} \sum _{i,j,k} D_j D_k \bb _i (x,s) \cdot D_k \left( \bb_j (x,s) \bb_i (x,s) \right) \, dx \, ds \notag \\   + 2 \int _{t_0} ^t (1+s)^{\alpha(r^*) + \delta + 1} \int _{\RR^3} \sum _{i,j,k} D_j D_k \uu _i (x,s) \cdot D_k \left( \uu_j (x,s) \ww_i (x,s) \right) \, dx \, ds \notag \\  + 2 \int _{t_0} ^t (1+s)^{\alpha(r^*) + \delta + 1} \int _{\RR^3} \sum _{i,j,k} D_j D_k \uu _i (x,s) \cdot D_k \left( \uu_j (x,s) \bb_i (x,s) \right) \, dx \, ds \notag \\  - 2 \int _{t_0} ^t (1+s)^{\alpha(r^*) + \delta + 1} \int _{\RR^3} \sum _{i,j,k} D_j D_k \bb _i (x,s) \cdot D_k \left( \bb_j (x,s) \uu_i (x,s) \right) \, dx \, ds. \notag
\end{aligned}
\end{align}
By Cauchy-Schwarz

\begin{align}
4 \chi \sum _{k = 1} ^3 \int _{t_0} ^t (1+s)^{\alpha(r^*) + \delta + 1}  \langle \nabla \times D_k \uu  (s), D_k \ww (s) \rangle \, ds \notag \\ \leq 4 \chi \int _{t_0} ^t (1+s)^{\alpha(r^*) + \delta + 1} \left( \|\nabla \ww(s) \|_{L^2}^2 + \|D^2 \uu (s)\|_{L^2}^2 \right) \, ds. \notag
\end{align}
For  $D_j D_k f_i (x,s) \cdot D_k \left( g_j (x,s) h_i (x,s) \right)$, where $f,g,h \in \{ \uu, \ww, \bb \}$ we have that

\begin{displaymath}
\int _{\RR^3} D_j D_k f _i (x,s) \cdot D_k \left( g_j (x,s) h_i (x,s) \right) \, dx \ \leq C \Vert \zz (s) \Vert _{L^{\infty}} \Vert \nabla \zz (s) \Vert _{L^2} \Vert D^2 \zz (s) \Vert _{L^2}.
\end{displaymath}
We then obtain

\begin{align*}
(1+t)^{\alpha(r^\ast) +\delta +1} \|\nabla \zz (t) \|_{L^2}^2 & + C \int_{t_0}^t (1+s)^{\alpha(r^\ast) +\delta +1} \| D^2 \zz(s) \|_{L^2}^2 ds \\ & \leq \int_{t_0}^t (1+s)^{\alpha(r^\ast) +\delta } \|\nabla \zz (s) \|_{L^2}^2 ds \\ & + C \int _{t_0}^t (1+s)^{\alpha(r^\ast) +\delta +1}\|\zz(s)\|_{L^\infty} \|\nabla\zz(s)\|_{L^2} \| D^2 \zz(s) \|_{L^2} ds \\ & \leq  \int_{t_0}^t (1+s)^{\alpha(r^\ast) +\delta } \|\nabla \zz (s) \|_{L^2}^2 ds \\ & + C \intop_{t_0}^t (1+s)^{\alpha(r^\ast) +\delta +1} \Vert \nabla \zz (s) \Vert _{L^2} ^{\frac{1}{2}} \| D^2 \zz(s) \|_{L^2}^2 ds,
\end{align*}
where we used that  $\|\zz \|_{L^\infty} \|\nabla \zz \|_{L^2} \leq \|\zz \|_{L^2}^{1/2} \|\nabla \zz \|^{1/2}_{L^2} \| D^2 \zz\|_{L^2}$, (see Kreiss, Hagstrom, Lorenz and Zingano \cite{MR1994780}) and that $\Vert \zz (t) \Vert _{L^2} \leq C$. For  some large enough $t_0$, we have that if $t > t_0$, then  $\Vert \nabla \zz (t) \Vert _{L^2}$ is small enough for the last term on the right hand side to absorbed by the second term in the left hand side. Hence

\begin{align}\label{2nd_order_inequality_done} 
(1+t)^{\alpha(r^\ast) +\delta +1} \|\nabla \zz (t) \|_{L^2}^2 & + \min\{\mu,\gamma,\nu\}\intop_{t_0}^t (1+s)^{\alpha(r^\ast) +\delta +1} \|\nabla^2 \zz(s) \|_{L^2}^2 ds \\ &  \leq C (1+t)^\delta \notag 
\end{align}
which is the result we wanted to prove. 
\end{proof}

\begin{proof}[Proof of Theorem \ref{decay-w}] As in the proof of Theorem \ref{decay-z}, we assume solutions are regular enough. We note that the equation for $\ww$ in (\ref{eqn:MHDmicropolar}) can be written as

\begin{displaymath}
\partial_t \ww   =  \mathbb{L} \ww  + \chi \nabla \times \uu - (\uu \cdot \nabla) \ww,
\end{displaymath}
where $\mathbb{L} = \gamma \Delta \ww +  \nabla (\nabla \cdot \ww) -2 \chi \ww$.  Let $t_0 > 0$ be as in Lemma \ref{gradient-z-decay}. Then, for $\mathcal{L} =  \gamma \Delta \ww +  \nabla (\nabla \cdot \ww)$, we have that

\begin{align}
\label{eqn:solucao-w}
\ww (x,t) & = e^{- 2 \chi (t - t_0)} e^{\mathcal{L} (t - t_0)} \ww (x, t_0) \notag \\ & - \int _{t_0} ^t e^{- 2 \chi (t - s)} e^{\mathcal{L} (t - s)} \left( \nabla \times \uu - (\uu \cdot \nabla) \ww \right) (x,s) \, ds.
\end{align}
First, as $\Vert e^{\mathcal{L}t}v \Vert _{L^2} \leq C \Vert v \Vert _{L^2}$, we have that

\begin{equation}
\label{eqn:firstestimatew}
\Vert e^{- 2 \chi (t - t_0)} e^{\mathcal{L} (t - t_0)} \ww (t_0) \Vert _{L^2} \leq C e^{- 2 \chi (t - t_0)}, \qquad C = C \left( \Vert \ww _0 \Vert _{L^2} \right).
\end{equation}
As $\Vert \nabla \times u \Vert _{L^2} \leq C \Vert \nabla u \Vert _{L^2}$, then

\begin{equation}
\label{eqn:secondestimatew}
\Vert e^{\mathcal{L} (t - s)} \nabla \times \uu (s)  \Vert _{L^2} \leq C \Vert \nabla \uu (s)  \Vert _{L^2} \leq C (1 + s) ^{- \frac{1}{2} \min \{ \frac{5}{2} + r^{\ast}, \frac{7}{2}\}}, 
\end{equation}
because of Lemma \ref{gradient-z-decay}. Also

\begin{align}
\label{eqn:thirdestimatew}
\Vert e^{\mathcal{L} (t - s)} (\uu \cdot \nabla) \ww (s)  \Vert _{L^2} & \leq \Vert e^{\Delta (t - s)} (\uu \cdot \nabla) \ww (s)  \Vert _{L^2} \leq \Vert e^{\Delta (t - s)} \Vert _{L^2} \Vert (\uu \cdot \nabla) \ww (s) \Vert _{L^1} \notag \\ & \leq \Vert e^{\Delta (t - s)} \Vert _{L^2} \Vert u(s)  \Vert _{L^2} \Vert \nabla \ww (s) \Vert _{L^2} \notag \\ & \leq C (t - s) ^{-\frac{3}{4}} (1 + s) ^{- \frac{1}{2} \min \{ \frac{5}{2} + r^{\ast}, \frac{7}{2}\}},
\end{align}
where we used (\ref{eqn:heat_estimate1})  and Lemma \ref{gradient-z-decay}. Now, from (\ref{eqn:solucao-w}) - (\ref{eqn:thirdestimatew}), we obtain

\begin{align}
\Vert \ww (t) \Vert _{L^2} & \leq C e^{- 2 \chi (t - t_0)} + \int_{t_0} ^t e^{- 2 \chi (t - s)} (1 + s) ^{- \min \{ \frac{5}{2} + r^{\ast}, \frac{7}{2}\}}  \left( 1 +  (t - s) ^{-\frac{3}{4}}  \right) \, ds  \notag \\ & \leq \tilde{C} (1 + t) ^{- \frac{1}{2} \min \{ \frac{5}{2} + r^{\ast}, \frac{7}{2}\}}, \quad \forall t>t_0 \notag.
\end{align}
We now conclude the proof, showing that the estimate also holds in $0 < t < t_0$. Indeed, letting $M =\max_{0\leq\tau\leq t_0}\{(1+\tau)^\kappa \|\ww(\tau) \|_{L^2} \}$, where we set   $\kappa  = \min \{ \frac{5}{4} + \frac{r^{\ast}}{2}, \frac{7}{4}\} $, we clearly get   $ \Vert \ww (t) \Vert _{L^2} \leq C (1+t) ^{-\kappa}$, for all  $t > 0$, with $ C =\max\{M,\tilde{C}\} $.
\end{proof}

\subsection{Proof of Theorem \ref{decay-difference}} We will need a gradient decay estimate.

\begin{Lemma}
\label{lemma-estimate-w-and-z}
		Let $\zz$ be a weak solution to (\ref{eqn:MHDmicropolar}), with $32 \, \chi (\mu + \chi + \gamma) > 1, \nu > 0$. Let $r^{\ast} (\zz_0) = r^{\ast}$ be the decay character of $\zz _0$, with $- \frac{3}{2} < r^{\ast} < \infty$. Then, 

\begin{align*}
 \Vert \nabla \ww (t) \Vert _{L^2} ^2 & + \Vert D^2 \zz (t) \Vert _{L^2} ^2 + (1+t)\Vert D^3 \zz (t) \Vert _{L^2} ^2 \notag \\ & \leq C (1 + t) ^{- \min \{ \frac{7}{2} + r^{\ast}, \frac{9}{2}\} }, \qquad \forall \, t > t_0,
\end{align*}		
for some $t_0 = t_0 (\Vert \zz_0 \Vert _{L^2})$.
	\end{Lemma}

\begin{proof} We follow the ideas in the proof of Lemma \ref{gradient-z-decay}. Differentiating \eqref{eqn:MHDmicropolar} with respect to $ x_{\ell_1} $ and $ x_{\ell_2} $, multiplying by $(1+t)^{{\alpha(r^\ast) + \delta + 2}} D_{\ell_1} D_{\ell_2} \left(\uu, \ww, \bb \right)$,   integrating the result on $ \RR^3 \times [t_0,t] $ and summing up we obtain

\begin{align*}
 (1+t)^{\alpha(r^\ast) + \delta + 2} \| D^2 \zz(t)\|_{L^2}^2 & + C \int _{t_0}^t  (1+s)^{\alpha(r^\ast) +\delta +2} \| D^3 \zz(s) \|_{L^2}^2 ds \\ & \leq C\intop_{t_0}^t  (1+s)^{\alpha(r^\ast) +\delta +1} \| D^2 \zz(s) \|_{L^2}^2 ds \\ & + C \int_{t_0}^t (s-t_0)^2 \| D^3 \zz(s)\|_{L^2} \|\zz(s)\|_{L^\infty} \| D^2 \zz(s)\|_{L^2} \, ds \\ & + C 
\int_{t_0}^t (s-t_0)^2 \| D^3 \zz(s)\|_{L^2} \| \nabla \zz(s)\|_{L^\infty(\RR^3)} \|\nabla \zz(s)\|_{L^2} \, ds  \\ & \leq  C \int _{t_0}^t  (1+s)^{\alpha(r^\ast) +\delta +1} \| D^2 \zz(s) \|_{L^2}^2 ds \\ &  + \int _{t_0}^t  (1+s)^{\alpha(r^\ast) +\delta +2} \|\zz(s) \|_{L^2}^{1/2} \|\nabla \zz(s) \|_{L^2}^{1/2} \| D^3 \zz(s)\|_{L^2} ds,
 \end{align*}   
where we used that $ \|\zz \|_{L^2} \|\nabla^2 \zz\|_{L^2} \leq C \|\zz \|_{L^2}^{1/2} \|\nabla \zz \|_{L^2}^{1/2} \|\nabla^3 \zz \|_{L^2}$. By \eqref{2nd_order_inequality_done}, we have
\begin{align}\label{3rd_order_inequality}
(1+t)^{\alpha(r^\ast) +\delta +2} \| D^2 \zz (t) \|_{L^2}^2 & + \min\{\mu,\gamma,\nu\}\intop_{t_0}^t (1+s)^{\alpha(r^\ast) +\delta +1} \| D^3 \zz(s) \|_{L^2}^2 ds \notag\\ &  \leq C (1+t)^\delta. 
\end{align}
Now, applying the same argument as in the previous Lemma (see also Guterres, Melo, Nunes and Perusato \cite{MR3906315}, for more details), we obtain
 \begin{equation*}
 (1+t)^{\alpha(r^\ast) + \delta + 3} \| D^3 \zz(t)\|_{L^2}^2 + C \int _{t_0}^t  (1+s)^{\alpha(r^\ast) +\delta +3} \| D^4 \zz(s) \|_{L^2}^2 ds \leq C(1+t)^\delta
 \end{equation*}
which yields
\begin{displaymath}
\Vert \nabla^2 \zz (t) \Vert _{L^2} ^2 + (1+t)\Vert \nabla^3 \zz (t) \Vert _{L^2} ^2 \leq C (1 + t) ^{- \min \{ \frac{7}{2} + r^{\ast}, \frac{9}{2}\} }.
\end{displaymath}
Furthermore, by standard Sobolev embeddings, we  obtain 
\begin{equation}\label{eqn:L4norm} 
   \|\nabla^j \zz(t)\|_{L^4}^2 \leq C(1+t)^{-\alpha(r^\ast) - 3/4 - j },
\end{equation}   
for all sufficiently large $t >0$   and  each $0 \leq j  \leq 2$. This particular bound allows us to estimate the gradient of the microrotational field. Indeed, for $\mathcal{L} =  \gamma \Delta +  \nabla (\nabla \cdot )$ we have
\begin{align}
\| \nabla\ww(t) \|_{L^2} & \leq e^{-2\,\chi\,(t-t_0)}\|\nabla e^{ C\mathcal{L}(t-t_0) } \zz(t_0) \|_{L^2}   \notag \\ & + \chi \int_{t_0}^t e^{-2\,\chi\,(t-s)} \|\nabla \,e^{C\mathcal{L}(t-s)} (\nabla \times \,\uu(s)\,) \|_{L^2} ds \notag \\ & + \int_{t_0}^t e^{-2\,\chi\,(t-s)  } \|e^{C\mathcal{L}(t-s)} ( \nabla(\uu\cdot \nabla\ww) (s) ) \|_{L^2} ds  \notag \\ & \leq C e^{-2\,\chi\,(t-t_0)}+ \chi\,\int_{t_0}^t e^{-2\,\chi\,(t-s)} \| \,e^{C\Delta(t-s)} \nabla^2\uu(s)\, \|_{L^2} ds \notag \\ & + C\,\int _{t_0}^t e^{-2\,\chi\,(t-s)  } \|e^{C\Delta(t-s)} \{\nabla(\uu\cdot \nabla\ww) (s) \}\|_{L^2} ds \notag \\ & \leq C e^{-2\,\chi\,(t-t_0)} + \chi \int_{t_0}^t e^{-2\,\chi\,(t-s)} \|  \nabla^2\uu(s)\, \|_{L^2} ds \notag \\ & + C\,\int _{t_0}^t e^{-2\,\chi\,(t-s)  } \| \{\nabla(\uu\cdot \nabla\ww) (s) \}\|_{L^2} ds. \notag
\end{align}    
By \eqref{eqn:L4norm} and Lemma \ref{gradient-z-decay}, we have

\begin{align}
\|\nabla \ww (t) \|_{L^2} & \leq C e^{-2\,\chi\,(t-t_0)}   + C\,\int _{t_0}^t e^{-2\,\chi\,(t-s)} (1+s)^{-\frac{\alpha(r^\ast)}{2} - 1} ds \notag \\ & + C\,\int _{t_0}^t e^{-2\,\chi\,(t-s)} \left( \|\zz(s) \|_{L^4} \|\nabla^2 \zz(s) \|_{L^4} + \|\nabla \zz(s) \|_{L^4}^2 \right) ds \notag \\ & \leq C e^{-2\,\chi\,(t-t_0)} + C\,\int _{t_0}^t e^{-2\,\chi\,(t-s)} (1+s)^{-\frac{\alpha(r^\ast)}{2} - 1} ds \notag \\ & + \int _{t_0}^t e^{-2\chi(t-s) } (1+s)^{- \alpha(r^\ast) - \frac{7}{4} } ds \notag \\ & \leq C e^{-2\,\chi\,(t-t_0)} + C\,\int _{t_0}^t e^{-2\,\chi\,(t-s)} (1+s)^{-\frac{\alpha(r^\ast)}{2} - 1} ds \notag \\ & \leq C(1+t)^{-\frac{\alpha(r^\ast)}{2} - 1} \notag
\end{align}
which concludes the proof. 
\end{proof}

\begin{proof}[Proof of Theorem \ref{decay-difference}]  We shall first prove \eqref{error_z}. For regular enough solutions, the  integral representation of $ \zz(t) $ and $ \bar{\zz}(t) $ and standard Leray's projector properties lead to 
\begin{equation}\label{eqn:z_Duhamel_error1}
\|\zz(t) - \bar{\zz}(t) \|_{L^2} \leq \intop_{t_0}^{t}  \|e^{\mathbb{A}\,(t - \tau)}  {\bf  \mathcal{Q}(\tau)}\|_{L^2} d\tau,
\end{equation}
where 
\begin{equation*}
\mathcal{Q}(t) = \left[	\begin{array}{c}
-(\uu \cdot \nabla) \uu(t)  \:\! + (\bb \cdot \nabla) \bb(t) \:\\ - (\uu \cdot \nabla) \ww (t)\:\\ -(\uu \cdot \nabla) \bb(t)\:\! + (\bb \cdot \nabla) \uu(t) \:\!
\end{array} \right] = \nabla \cdot \left[ \begin{array}{c}
-\uu \otimes \uu(t) + \bb \otimes \bb(t) \\ - \uu \otimes \ww (t) \\ - \uu \otimes \bb(t)\:\! + \bb \otimes \uu(t) 
\end{array} \right]
\end{equation*} 
and 	$ \bigl(e^{\mathbb{A}\,t}\bigr)_{t \geq t_0} $ is the semigroup associated to the linear system \eqref{eqn:linear}. Using Lemma \ref{lemma-eigenvalue} and Plancherel identity, we obtain
\begin{equation}\label{eqn:semigroup_estimate}
\intop_{t_0}^{t}  \|e^{\mathbb{A}\,(t - \tau)}  {\bf  \mathcal{Q}(\tau)}\|_{L^2} d\tau \leq \intop_{t_0}^{t}  \|e^{C\Delta\,(t - \tau)}  {\bf  \mathcal{Q}(\tau)}\|_{L^2} d\tau.
\end{equation}
Using the heat kernel estimate (\ref{eqn:heat_estimate1}), we estimate the convolution term in two different ways, namely

\begin{align}
\label{eqn:first-part-nonlinear-term}
\Vert e^{\Delta (t - \tau)}  {\bf  \mathcal{Q}(\tau)} \Vert _{L^2} &  = \Vert \nabla e^{\Delta (t - \tau)} \mathcal{NL} (\uu, \ww, \bb) (\tau) \Vert _{L^2} \notag \\ & \leq C \Vert \nabla e^{\Delta (t - \tau)}  \Vert _{L^2} \Vert  \mathcal{NL} (\uu, \ww, \bb) (\tau) \Vert _{L^1} \notag \\ & \leq C (t-\tau)^{-\frac{5}{4}}  \|\zz(\tau) \|^2_{L^2}
\end{align}
where $\mathcal{Q}(t) = \nabla \mathcal{NL} (\uu, \ww, \bb)$, and 

\begin{align}
\label{eqn:second-part-nonlinear-term}
\Vert e^{\Delta (t - \tau)}  {\bf  \mathcal{Q}(\tau)} \Vert _{L^2} &  \leq C  \Vert e^{\Delta (t - \tau)}  \Vert _{L^2} \Vert  \mathcal{Q}(\tau) \Vert _{L^1} \notag \\ & C \leq (t-\tau)^{-3/4}  \|\zz(\tau) \|_{L^2} \|\nabla \zz(\tau) \|_{L^2}
\end{align}
Then \eqref{eqn:z_Duhamel_error1} - \eqref{eqn:second-part-nonlinear-term} imply 
 \begin{align*}
\|\zz(t) - \bar{\zz}(t) \|_{L^2} & \leq \intop_{t_0}^{\frac{t_0 +t}{2}}  (t-\tau)^{-\frac{5}{4}}  \|\zz(\tau) \|^2_{L^2} d\tau + \intop_{\frac{t_0 + t}{2}}^{t}  (t-\tau)^{-\frac{3}{4}}  \|\zz(\tau) \|_{L^2} \|\nabla \zz(\tau) \|_{L^2} d\tau \notag \\ & =  \mathcal{I}_1(t) + \mathcal{I}_2(t).
 \end{align*} 
Using Theorem \ref{decay-z}, a straightforward calculation leads to 

\begin{displaymath}
\mathcal{I}_1(t) \leq C (1 + t) ^{- \left( \frac{7}{4} + r^{\ast} \right)}, \text{\, for } r^{\ast} <  - \frac{1}{2}, \qquad \mathcal{I}_1(t) \leq C (1 + t) ^{- \frac{5}{4}}, \text{\, for } r^{\ast} >  - \frac{1}{2}.
\end{displaymath}
For $r^{\ast} = - \frac{1}{2}$, we have that

\begin{displaymath}
\mathcal{I}_1(t) \leq C (1 + t) ^{- \frac{5}{4} + \delta}, \qquad \forall \, \delta > 0.
\end{displaymath}
Analogously, from Theorem \ref{decay-z} and Lemma \ref{gradient-z-decay}

\begin{displaymath}
\mathcal{I}_2(t) \leq C (1 + t) ^{- \min \{\frac{15}{4} + 2 r^{\ast}, \frac{23}{4} \}}, \qquad -\frac{3}{2} < r^{\ast} < \infty.
\end{displaymath}
As a result of this, the estimate is true for all $t > t_0$. An argument similar to that in page 239 in Kreiss, Hagstrom, Lorenz and Zingano \cite{MR1994780} allows to extend the bound to all $t > 0$. 

 In order to prove \eqref{error_w}, we will need the following estimate.

 \begin{Lemma} 
\label{lemma-gradient-z-barz} 
 Let $\zz$ be a weak solution to (\ref{eqn:MHDmicropolar}) , with $32 \, \chi (\mu + \chi + \gamma) > 1, \nu > 0$. Let $r^{\ast} (\zz_0) = r^{\ast}$ be the decay character of $\zz _0$, with $- \frac{3}{2} < r^{\ast} < \infty$. Then,  
 	\begin{displaymath}
 	\Vert \nabla\zz (t) - \nabla\bar{{\zz}}(t) \Vert _{L^2} ^2 \leq C (1 + t) ^{- \min\{\frac{9}{4} + 2r^\ast,\,\frac{7}{4}\}}, \qquad \forall t >t_0,
 	\end{displaymath}
 	for some $t_0 = t_0 (\Vert \zz_0 \Vert _{L^2})$.
\end{Lemma}

\begin{proof}[Proof of Lemma \ref{lemma-gradient-z-barz}] We have that
\begin{equation*}
\| \nabla \zz(t) - \nabla \bar{\zz}(t) \|_{L^2} \leq \intop_{t_0}^{t}  \| \nabla e^{\mathbb{A}\,(t - \tau)}  {\bf  \mathcal{Q}(\tau)}\|_{L^2} d\tau.
\end{equation*}
As before
\begin{align*}
\Vert\nabla  e^{\Delta (t - \tau)}  {\bf  \mathcal{Q}(\tau)} \Vert _{L^2} &  = \Vert \nabla ^2 e^{\Delta (t - \tau)} \mathcal{NL} (\uu, \ww, \bb) (\tau) \Vert _{L^2} \notag \\ & \leq C \Vert \nabla^2 e^{\Delta (t - \tau)}  \Vert _{L^2} \Vert  \mathcal{NL} (\uu, \ww, \bb) (\tau) \Vert _{L^1} \notag \\ & \leq C (t-\tau)^{-\frac{7}{4}}  \|\zz(\tau) \|^2_{L^2}
\end{align*}
where $\mathcal{Q}(t) = \nabla \mathcal{NL} (\uu, \ww, \bb)$, and 

\begin{align*}
\Vert \nabla e^{\Delta (t - \tau)}  {\bf  \mathcal{Q}(\tau)} \Vert _{L^2} &  \leq C  \Vert \nabla e^{\Delta (t - \tau)}  \Vert _{L^{\frac{4}{3}}} \Vert  \mathcal{Q}(\tau) \Vert _{L^{\frac{4}{3}}} \notag \\ &  \leq C (t-\tau)^{- \frac{7}{8}}  \|\zz(\tau) \|_{L^4} \|\nabla \zz(\tau) \|_{L^2} \\ & \leq C \leq (t-\tau)^{- \frac{7}{8}}  \|\zz(\tau) \|_{L^2} ^{\frac{1}{4}} \|\nabla \zz(\tau) \|_{L^2} ^{\frac{7}{4}}.
\end{align*}
Then

\begin{align*}
\|\nabla \zz(t) - \nabla \bar{\zz} (t) \|_{L^2} & \leq \int _{t_0}^{\frac{t_0 + t}{2}}  (t-\tau)^{-\frac{7}{4}}  \|\zz(\tau) \|^2_{L^2} d\tau + \int _{\frac{t_0 + t}{2}}^{t}  (t-\tau)^{-\frac{7}{8}}  \|\zz(\tau)\cdot \nabla \zz(\tau)\|_{L^{{\frac{4}{3}}}}  d\tau \\ & \leq \int _{t_0}^{\frac{t_0 + t}{2}}  (t-\tau)^{-\frac{7}{4}}  \|\zz(\tau) \|^2_{L^2} d\tau + \int _{\frac{t_0 + t}{2}}^{t}  (t-\tau)^{-\frac{7}{8}}  \|\zz(\tau) \|_{L^2} ^{\frac{1}{4}} \|\nabla \zz(\tau) \|_{L^2} ^{\frac{7}{4}}  d\tau \\ & =I_1(t) +I_2(t)
\end{align*}
As in the previous Lemma, from Theorem \ref{decay-z} we obtain 

\begin{displaymath}
I_1(t) \leq C (1 + t) ^{- \left( \frac{9}{4} + r^{\ast} \right)}, \text{\, for } r^{\ast} <  - \frac{1}{2}, \qquad I_1(t) \leq C (1 + t) ^{- \frac{7}{4}}, \text{\, for } r^{\ast} >  - \frac{1}{2}.
\end{displaymath}
For $r^{\ast} = - \frac{1}{2}$, we have that

\begin{displaymath}
I_1(t) \leq C (1 + t) ^{- \frac{7}{4} + \delta}, \qquad \forall \, \delta > 0.
\end{displaymath}
From Theorem \ref{decay-z} and Lemma \ref{gradient-z-decay}

\begin{displaymath}
I_2(t) \leq C (1 + t) ^{- \min \{\frac{13}{4} + r^{\ast}, \frac{17}{4} \}}, \qquad -\frac{3}{2} < r^{\ast} < \infty.
\end{displaymath}
Then, the decay is given by $I_1 (t)$ $ \forall t>t_0 $.
\end{proof}

We return to the proof of (\ref{error_w}).  By using Duhamel's principle, we take advantage of the damping term $ -2\chi\,\ww $ to get, after a few computations

\begin{align*}
\|\ww(t) - \bar{\ww}(t) \|_{L^2} & \leq\chi\,\int_{t_0}^t e^{-2\,\chi\,(t-s)  } \|e^{C\mathcal{L}(t-s)} \{\nabla \times (\,\uu(s)- \bar{\uu}(s)\,) \}\|_{L^2} ds \\ & + \,\int _{t_0}^t e^{-2\,\chi\,(t-s)  } \|e^{C\mathcal{L}(t-s)} \{(\uu\cdot \nabla\ww) (s) \}\|_{L^2} ds \\ & \leq \chi\,\int _{t_0}^t e^{-2\,\chi\,(t-s)  } \|e^{C\Delta(t-s)} \{\nabla \times (\,\uu(s)- \bar{\uu}(s)\,) \}\|_{L^2} ds \\ & + \,\int_{t_0}^t e^{-2\,\chi\,(t-s)  } \|e^{C\Delta(t-s)} \{(\uu\cdot \nabla\ww) (s) \}\|_{L^2} ds \\ & =  \mathcal{J}_1(t) + \mathcal{J}_2(t).
\end{align*}

If $ -\frac{3}{2} < r^\ast < -\frac{1}{2} $, by Lemma \ref{lemma-gradient-z-barz} we have
\begin{align*}
\mathcal{J}_1(t) & \leq \int _{t_0}^{(t_0+t)/2} e^{-2\,\chi\,(t-s)  } (1+s)^{-\frac{9}{4} - r^\ast} ds 
+ \int _{(t_0 + t)/2}^t  e^{-2\,\chi\,(t-s)  } (1+s)^{-\frac{9}{4} - r^\ast} ds  \\ & = \mathcal{J}_{11}(t) + \mathcal{J}_{12}(t).
\end{align*} 
We now observe that $ \mathcal{J}_{11}(t) \leq  C e^{-\chi\,t}$ and $  \mathcal{J}_{12}(t) \leq  C (1+t)^{-\frac{9}{4}-r^\ast}$. For the term $ \mathcal{J}_2(t) $, we proceed as follows.
By using the Lemma \ref{lemma-estimate-w-and-z}, we have 
\begin{align*}
\mathcal{J}_{2}(t) & \leq \int _{t_0}^{(t_0 +t)/2}  e^{-2\,\chi\,(t-s)}(t-s)^{-\frac{3}{4}} (1+s)^{-\frac{5}{2} - r^\ast} ds \\ & + \int _{(t_0 + t)/2}^{t} e^{-2\,\chi\,(t-s)} (t-s)^{-\frac{3}{4}} (1+s)^{-\frac{5}{2} - r^\ast} ds \\ & = \mathcal{J}_{21}(t)+\mathcal{J}_{22}(t) .
\end{align*} 
But $ \mathcal{J}_{21}(t) \leq C e^{-\chi\,t} t^{-3/4}$ and 
\begin{equation*}
\mathcal{J}_{22}(t) \leq C (1+t)^{-\frac{5}{2} -r^\ast} \intop_{(t_0 +t)/2}^t e^{-2\,\chi(t-s)} (t-s)^{-3/4} ds \leq C (1+t)^{-\frac{5}{2} -r^\ast} \Gamma(1/4), 
\end{equation*}
where $ \Gamma $ is the  Gamma function. Hence, $ \mathcal{J}_1(t) + \mathcal{J}_2(t) \leq C(1+t)^{-\frac{9}{4} - r^\ast} $.
If  $   - \frac{1}{2} \leq r^{\ast} < 1 $, we similarly obtain $\mathcal{J}_1(t)\leq e^{-\chi\,t}  + C(1+t)^{-\frac{7}{4}}$ and $ \mathcal{J}_2 (t) \leq C e^{-\chi\,t}t^{-3/4} + C(1+t)^{-\frac{5}{2} - r^\ast} $which leads to $ \mathcal{J}_1(t) + \mathcal{J}_2(t) \leq C(1+t)^{-\frac{7}{4}}  $, since $ - \frac{1}{2} \leq r^{\ast} < 1  $. Finally, when $ r^\ast \geq 1 $, we immediately get $ \mathcal{J}_1(t) + \mathcal{J}_2(t) \leq  C(1+t)^{-\frac{7}{4}}$ 
which concludes the proof for $ t>t_0 $ and repeating the argument given in the proof of Theorem \ref{decay-w} or (\ref{error_z}), the upper bound  holds also for $ 0<t<t_0 $ concluding the demonstration.    
\end{proof}

\bibliographystyle{plain}
\bibliography{NichePerusatoBiblio}{}

\begin{thebibliography}{10}

\bibitem{MR0443550}
G.~Ahmadi and M.~Shahinpoor.
\newblock Universal stability of magneto-micropolar fluid motions.
\newblock {\em Internat. J. Engrg. Sci.}, 12:657--663, 1974.

\bibitem{MR2493562}
Clayton Bjorland and Mar\'{\i}a~E. Schonbek.
\newblock Poincar\'e's inequality and diffusive evolution equations.
\newblock {\em Adv. Differential Equations}, 14(3-4):241--260, 2009.

\bibitem{MR2646523}
J.~L. Boldrini, M.~Dur\'{a}n, and M.~A. Rojas-Medar.
\newblock Existence and uniqueness of strong solution for the incompressible
  micropolar fluid equations in domains of {$\Bbb R^3$}.
\newblock {\em Ann. Univ. Ferrara Sez. VII Sci. Mat.}, 56(1):37--51, 2010.

\bibitem{MR3493117}
Lorenzo Brandolese.
\newblock Characterization of solutions to dissipative systems with sharp
  algebraic decay.
\newblock {\em SIAM J. Math. Anal.}, 48(3):1616--1633, 2016.

\bibitem{MR3955606}
P.~Braz~e Silva, F.~W. Cruz, L.~B.~S. Freitas, and P.~R. Zingano.
\newblock On the {$L^2$} decay of weak solutions for the 3{D} asymmetric fluids
  equations.
\newblock {\em J. Differential Equations}, 267(6):3578--3609, 2019.

\bibitem{MR3516831}
P.~Braz~e Silva, L.~Friz, and M.~A. Rojas-Medar.
\newblock Exponential stability for magneto-micropolar fluids.
\newblock {\em Nonlinear Anal.}, 143:211--223, 2016.

\bibitem{MR3592792}
Pablo Braz~e Silva, Wilberclay~G. Melo, and Paulo~R. Zingano.
\newblock Lower bounds on blow-up of solutions for magneto-micropolar fluid
  systems in homogeneous {S}obolev spaces.
\newblock {\em Acta Appl. Math.}, 147:1--17, 2017.

\bibitem{MR673830}
Luis Caffarelli, Robert Kohn, and Louis Nirenberg.
\newblock Partial regularity of suitable weak solutions of the
  {N}avier-{S}tokes equations.
\newblock {\em Comm. Pure Appl. Math.}, 35(6):771--831, 1982.

\bibitem{CruzNovais}
Felipe Cruz and Michele Novais.
\newblock Optimal {$L^2$} decay of the magneto-micropolar system in
  {$\mathbb{R}^3$}.
\newblock {\em Z. Angew. Math. Phys.}, pages 71--91, 2020.

\bibitem{MR0204005}
A.~Cemal Eringen.
\newblock Theory of micropolar fluids.
\newblock {\em J. Math. Mech.}, 16:1--18, 1966.

\bibitem{MR2324348}
L.~C.~F. Ferreira and E.~J. Villamizar-Roa.
\newblock Micropolar fluid system in a space of distributions and large time
  behavior.
\newblock {\em J. Math. Anal. Appl.}, 332(2):1425--1445, 2007.

\bibitem{MR3565380}
Lucas C.~F. Ferreira, C\'{e}sar~J. Niche, and Gabriela Planas.
\newblock Decay of solutions to dissipative modified quasi-geostrophic
  equations.
\newblock {\em Proc. Amer. Math. Soc.}, 145(1):287--301, 2017.

\bibitem{MR2639150}
Sadek Gala.
\newblock Regularity criteria for the 3{D} magneto-micropolar fluid equations
  in the {M}orrey-{C}ampanato space.
\newblock {\em NoDEA Nonlinear Differential Equations Appl.}, 17(2):181--194,
  2010.

\bibitem{MR2945855}
Sadek Gala, Yoshihiro Sawano, and Hitoshi Tanaka.
\newblock A new {B}eale-{K}ato-{M}ajda criteria for the 3{D} magneto-micropolar
  fluid equations in the {O}rlicz-{M}orrey space.
\newblock {\em Math. Methods Appl. Sci.}, 35(11):1321--1334, 2012.

\bibitem{MR0467030}
Giovanni~P. Galdi and Salvatore Rionero.
\newblock A note on the existence and uniqueness of solutions of the micropolar
  fluid equations.
\newblock {\em Internat. J. Engrg. Sci.}, 15(2):105--108, 1977.

\bibitem{MR2927090}
Congchong Guo, Zujin Zhang, and Jialin Wang.
\newblock Regularity criteria for the 3{D} magneto-micropolar fluid equations
  in {B}esov spaces with negative indices.
\newblock {\em Appl. Math. Comput.}, 218(21):10755--10758, 2012.

\bibitem{MR3906315}
R.~H. Guterres, W.~G. Melo, J.~R Nunes, and C.~F. Perusato.
\newblock On the large time decay of asymmetric flows in homogeneous {S}obolev
  spaces.
\newblock {\em J. Math. Anal. Appl.}, 471(1-2):88--101, 2019.

\bibitem{MR3853142}
Robert~H. Guterres, Juliana~R. Nunes, and Cilon~F. Perusato.
\newblock Decay rates for the magneto-micropolar system in {$L^2(\Bbb R^n)$}.
\newblock {\em Arch. Math. (Basel)}, 111(4):431--442, 2018.

\bibitem{Kato1984}
Tosio Kato.
\newblock Strong {$L ^p$}-solutions of the {N}avier-{S}tokes equation in {$\R
  ^m$}, with applications to weak solutions.
\newblock {\em Mathematische Zeitschrift}, 187:471--480, 1984.

\bibitem{MR1994780}
Heinz-Otto Kreiss, Thomas Hagstrom, Jens Lorenz, and Paulo Zingano.
\newblock Decay in time of incompressible flows.
\newblock {\em J. Math. Fluid Mech.}, 5(3):231--244, 2003.

\bibitem{MR2047286}
Grzegorz \L~ukaszewicz and Witold Sadowski.
\newblock Uniform attractor for 2{D} magneto-micropolar fluid flow in some
  unbounded domains.
\newblock {\em Z. Angew. Math. Phys.}, 55(2):247--257, 2004.

\bibitem{MR1938147}
Pierre-Gilles Lemari{\'e}-Rieusset.
\newblock {\em Recent developments in the {N}avier-{S}tokes problem}, volume
  431 of {\em Chapman \& Hall/CRC Research Notes in Mathematics}.
\newblock Chapman \& Hall/CRC, Boca Raton, FL, 2002.

\bibitem{MR3825173}
Ming Li and Haifeng Shang.
\newblock Large time decay of solutions for the 3{D} magneto-micropolar
  equations.
\newblock {\em Nonlinear Anal. Real World Appl.}, 44:479--496, 2018.

\bibitem{MR3758697}
Liangliang Ma.
\newblock Global existence of three-dimensional incompressible
  magneto-micropolar system with mixed partial dissipation, magnetic diffusion
  and angular viscosity.
\newblock {\em Comput. Math. Appl.}, 75(1):170--186, 2018.

\bibitem{MR3995955}
Exequiel Mallea-Zepeda and Elva Ortega-Torres.
\newblock Control problem for a magneto-micropolar flow with mixed boundary
  conditions for the velocity field.
\newblock {\em J. Dyn. Control Syst.}, 25(4):599--618, 2019.

\bibitem{MR3429636}
Wilberclay~G. Melo.
\newblock The magneto-micropolar equations with periodic boundary conditions:
  solution properties at potential blow-up times.
\newblock {\em J. Math. Anal. Appl.}, 435(2):1194--1209, 2016.

\bibitem{MR3355116}
C\'esar~J. Niche and Mar\'\i a~E. Schonbek.
\newblock Decay characterization of solutions to dissipative equations.
\newblock {\em J. Lond. Math. Soc. (2)}, 91(2):573--595, 2015.

\bibitem{MR3125139}
Piotr Orli\'{n}ski.
\newblock The existence of an exponential attractor in magneto-micropolar fluid
  flow via the {$\ell$}-trajectories method.
\newblock {\em Colloq. Math.}, 132(2):221--238, 2013.

\bibitem{MR2879801}
E.~E. Ortega-Torres, M.~A. Rojas-Medar, and R.~C. Cabrales.
\newblock A uniform error estimate in time for spectral {G}alerkin
  approximations of the magneto-micropolar fluid equations.
\newblock {\em Numer. Methods Partial Differential Equations}, 28(2):689--706,
  2012.

\bibitem{MR1810322}
Elva~E. Ortega-Torres and Marko~A. Rojas-Medar.
\newblock Magneto-micropolar fluid motion: global existence of strong
  solutions.
\newblock {\em Abstr. Appl. Anal.}, 4(2):109--125, 1999.

\bibitem{doi:10.1080/00036811.2019.1578347}
C.~F. Perusato, W.~G. Melo, R.~H. Guterres, and J.~R. Nunes.
\newblock Time asymptotic profiles to the magneto-micropolar system.
\newblock {\em Applicable Analysis}, 0(0):1--14, 2019.

\bibitem{MR1479160}
M.~A. Rojas-Medar.
\newblock Magneto-micropolar fluid motion: on the convergence rate of the
  spectral {G}alerkin approximations.
\newblock {\em Z. Angew. Math. Mech.}, 77(10):723--732, 1997.

\bibitem{MR1484679}
Marko~A. Rojas-Medar.
\newblock Magneto-micropolar fluid motion: existence and uniqueness of strong
  solution.
\newblock {\em Math. Nachr.}, 188:301--319, 1997.

\bibitem{MR1666509}
Marko~A. Rojas-Medar and Jos\'{e}~Luiz Boldrini.
\newblock Magneto-micropolar fluid motion: existence of weak solutions.
\newblock {\em Rev. Mat. Complut.}, 11(2):443--460, 1998.

\bibitem{MR1960745}
Witold Sadowski.
\newblock Upper bound for the number of degrees of freedom for
  magneto-micropolar flows and turbulence.
\newblock {\em Internat. J. Engrg. Sci.}, 41(8):789--800, 2003.

\bibitem{MR571048}
Mar\'{\i}a~E. Schonbek.
\newblock Decay of solutions to parabolic conservation laws.
\newblock {\em Comm. Partial Differential Equations}, 5(7):449--473, 1980.

\bibitem{MR775190}
Mar\'{\i}a~E. Schonbek.
\newblock {$L^2$} decay for weak solutions of the {N}avier-{S}tokes equations.
\newblock {\em Arch. Rational Mech. Anal.}, 88(3):209--222, 1985.

\bibitem{MR837929}
Mar\'{\i}a~E. Schonbek.
\newblock Large time behaviour of solutions to the {N}avier-{S}tokes equations.
\newblock {\em Comm. Partial Differential Equations}, 11(7):733--763, 1986.

\bibitem{MR3912713}
Zhong Tan, Wenpei Wu, and Jianfeng Zhou.
\newblock Global existence and decay estimate of solutions to
  magneto-micropolar fluid equations.
\newblock {\em J. Differential Equations}, 266(7):4137--4169, 2019.

\bibitem{MR3041770}
Yinxia Wang.
\newblock Regularity criterion for a weak solution to the three-dimensional
  magneto-micropolar fluid equations.
\newblock {\em Bound. Value Probl.}, pages 2013:58, 12, 2013.

\bibitem{MR3369594}
Yinxia Wang.
\newblock Blow-up criteria of smooth solutions to the three-dimensional
  magneto-micropolar fluid equations.
\newblock {\em Bound. Value Probl.}, pages 2015:118, 10, 2015.

\bibitem{MR3990123}
Yinxia Wang and Liuxin Gu.
\newblock Global regularity of 3{D} magneto-micropolar fluid equations.
\newblock {\em Appl. Math. Lett.}, 99:105980, 9, 2020.

\bibitem{MR3543126}
Yinxia Wang and Keyan Wang.
\newblock Global well-posedness of 3{D} magneto-micropolar fluid equations with
  mixed partial viscosity.
\newblock {\em Nonlinear Anal. Real World Appl.}, 33:348--362, 2017.

\bibitem{MR2775733}
Yu-Zhu Wang, Liping Hu, and Yin-Xia Wang.
\newblock A {B}eale-{K}ato-{M}adja criterion for magneto-micropolar fluid
  equations with partial viscosity.
\newblock {\em Bound. Value Probl.}, pages Art. ID 128614, 14, 2011.

\bibitem{MR2834313}
Yu-Zhu Wang, Yifang Li, and Yin-Xia Wang.
\newblock Blow-up criterion of smooth solutions for magneto-micropolar fluid
  equations with partial viscosity.
\newblock {\em Bound. Value Probl.}, pages 2011:11, 11, 2011.

\bibitem{MR881519}
Michael Wiegner.
\newblock Decay results for weak solutions of the {N}avier-{S}tokes equations
  on {${\bf R}\sp n$}.
\newblock {\em J. London Math. Soc. (2)}, 35(2):303--313, 1987.

\bibitem{MR3017165}
Zhaoyin Xiang and Huizhi Yang.
\newblock On the regularity criteria for the 3{D} magneto-micropolar fluids in
  terms of one directional derivative.
\newblock {\em Bound. Value Probl.}, pages 2012:139, 14, 2012.

\bibitem{MR2158216}
Norikazu Yamaguchi.
\newblock Existence of global strong solution to the micropolar fluid system in
  a bounded domain.
\newblock {\em Math. Methods Appl. Sci.}, 28(13):1507--1526, 2005.

\bibitem{MR3310628}
Kazuo Yamazaki.
\newblock 3-{D} stochastic micropolar and magneto-micropolar fluid systems with
  non-{L}ipschitz multiplicative noise.
\newblock {\em Commun. Stoch. Anal.}, 8(3):413--437, 2014.

\bibitem{MR3622437}
Kazuo Yamazaki.
\newblock Exponential convergence of the stochastic micropolar and
  magneto-micropolar fluid systems.
\newblock {\em Commun. Stoch. Anal.}, 10(3):Article 2, 271--295, 2016.

\bibitem{MR3810101}
Kazuo Yamazaki.
\newblock Large deviation principle for the micropolar, magneto-micropolar
  fluid systems.
\newblock {\em Discrete Contin. Dyn. Syst. Ser. B}, 23(2):913--938, 2018.

\bibitem{MR3903776}
Kazuo Yamazaki.
\newblock Gibbsian dynamics and ergodicity of stochastic micropolar fluid
  system.
\newblock {\em Appl. Math. Optim.}, 79(1):1--40, 2019.

\bibitem{MR4034672}
Kazuo Yamazaki.
\newblock Irreducibility of the three, and two and a half dimensional
  {H}all-magnetohydrodynamics system.
\newblock {\em Phys. D}, 401:132199, 21, 2020.

\bibitem{MR2778615}
Baoquan Yuan.
\newblock Regularity of weak solutions to magneto-micropolar fluid equations.
\newblock {\em Acta Math. Sci. Ser. B (Engl. Ed.)}, 30(5):1469--1480, 2010.

\bibitem{MR4008703}
Baoquan Yuan and Xiao Li.
\newblock Regularity of weak solutions to the 3{D} magneto-micropolar equations
  in {B}esov spaces.
\newblock {\em Acta Appl. Math.}, 163:207--223, 2019.

\bibitem{MR2419091}
Jia Yuan.
\newblock Existence theorem and blow-up criterion of the strong solutions to
  the magneto-micropolar fluid equations.
\newblock {\em Math. Methods Appl. Sci.}, 31(9):1113--1130, 2008.

\bibitem{MR3001674}
Hui Zhang and Yongye Zhao.
\newblock Blow-up criterion for strong solutions to the 3{D} magneto-micropolar
  fluid equations in the multiplier space.
\newblock {\em Electron. J. Differential Equations}, pages No. 188, 7, 2012.

\bibitem{MR2781751}
Zujin Zhang, Zheng-an Yao, and Xiaofeng Wang.
\newblock A regularity criterion for the 3{D} magneto-micropolar fluid
  equations in {T}riebel-{L}izorkin spaces.
\newblock {\em Nonlinear Anal.}, 74(6):2220--2225, 2011.

\end{thebibliography}

\end{document}